\documentclass[a4paper,12pt]{amsart}
\usepackage{environ}
\usepackage{fancyhdr}
\usepackage{mabliautoref}
\usepackage{amssymb,amsthm,amsmath}
\RequirePackage[dvipsnames,usenames]{xcolor}
\usepackage{hyperref}
\usepackage{cleveref}
\usepackage{mathtools}
\usepackage{tcolorbox}
\usepackage[all]{xy}
\usepackage{tikz}
\usepackage{tikz-cd}
\usetikzlibrary{decorations.markings,shapes,positioning}
\usepackage{enumitem}
\usepackage{stmaryrd}
\usepackage{xfrac}
\usepackage{fix-cm}
\usepackage{faktor}
\usepackage{minibox}
\usepackage{commath}
\usepackage{chngcntr}
\usepackage{setspace}
\usepackage{array}
\usepackage[toc,page]{appendix}
\usepackage{calligra,mathrsfs}
\usepackage{mathtools}
\usepackage{bbm}

\hypersetup{
	bookmarks,
	bookmarksdepth=3,
	bookmarksopen,
	bookmarksnumbered,
	pdfstartview=FitH,
	colorlinks,backref,hyperindex,
	linkcolor=Sepia,
	anchorcolor=BurntOrange,
	citecolor=Cyan,
	filecolor=BlueViolet,
	menucolor=Yellow,
	urlcolor=OliveGreen
}

\newtcolorbox{activitybox}[1][]{%
	breakable,
	enhanced,
	colback=lightergray,
	boxrule=3pt,
	arc=5pt,
	outer arc=5pt,
	boxsep=10pt,
	colframe=darkergray,
	coltitle=white,
	#1
}

\newcommand{\quot}[2]{%
	\raise1ex\hbox{$#1$}\Big/\lower1ex\hbox{$#2$}%
}

\newcommand{\colim}{\varinjlim}
\renewcommand{\lim}{\varprojlim}

\makeatletter\def\Rvarlim@#1#2{%
	\vtop{\m@th\ialign{##\cr
			\hfil$#1\operator@font Rlim$\hfil\cr
			\noalign{\nointerlineskip\kern1.5\ex@}#2\cr
			\noalign{\nointerlineskip\kern-\ex@}\cr}}%
}
\makeatother

\makeatletter\def\Rlim{%
	\mathop{\mathpalette\Rvarlim@{\leftarrowfill@\textstyle}}\nmlimits@
}
\makeatother

\newcommand{\expl}[2]{\underset{\mathclap{\minibox[c]{$\uparrow$\\ \fbox{\footnotesize #2}}}}{#1}}

\newcommand{\ra}{\rightarrow}
\newcommand{\dra}{\dashrightarrow}

\def\xto#1#2{\xrightarrow{#1}{#2}}
\def\xto#1{\xrightarrow{#1}}
\newcommand{\surj}{\twoheadrightarrow}
\newcommand{\inj}{\hookrightarrow}
\newcommand{\ttilde}{\widetilde}

\newcommand{\HHom}{\mathcal{H}om}
\newcommand{\EExt}{\mathcal{E}xt}
\newcommand{\et}{\acute{e}t}

\newcommand{\stacksproj}[1]{{\cite[Tag~\href{http://stacks.math.columbia.edu/tag/#1}{#1}]{Stacks_Project}}}

\newcommand{\inc}{\subseteq}

\newcommand{\esp}{\mbox{ }}

\DeclareMathAlphabet{\mathchanc}{OT1}{pzc}%
{m}{it}

\newcommand{\bigset}[2]{\big\{ \ #1 \  \big| \ #2 \ \big\}}

\newcommand{\bA}{\mathbb{A}}

\newcommand{\bD}{\mathbb{D}}

\newcommand{\bF}{\mathbb{F}}

\newcommand{\bZ}{\mathbb{Z}}

\newcommand{\scr}{\mathcal}
\newcommand{\cA}{\scr{A}}
\newcommand{\cB}{\scr{B}}
\newcommand{\cC}{\scr{C}}

\newcommand{\cE}{\scr{E}}
\newcommand{\cF}{\scr{F}}
\newcommand{\cG}{\scr{G}}
\newcommand{\cH}{\scr{H}}
\newcommand{\cI}{\scr{I}}

\newcommand{\cM}{\scr{M}}
\newcommand{\cN}{\scr{N}}
\newcommand{\cO}{\scr{O}}
\newcommand{\cP}{\scr{P}}
\newcommand{\cQ}{\scr{Q}}
\newcommand{\cR}{\scr{R}}

\newcommand{\cV}{\scr{V}}

\DeclareMathOperator{\Nil}{{Nil}}

\DeclareMathOperator{\ind}{{ind}}

\DeclareMathOperator{\ev}{{ev}}

\DeclareMathOperator{\Hom}{Hom}

\DeclareMathOperator{\im}{{im}}

\DeclareMathOperator{\Spec}{{Spec}}

\DeclareMathOperator{\Supp}{{Supp}}

\DeclareMathOperator{\Perv}{Perv}

\DeclareMathOperator{\Crys}{Crys}
\DeclareMathOperator{\coh}{coh}

\DeclareMathOperator{\unit}{unit}

\DeclareMathOperator{\Mod}{Mod}

\DeclareMathOperator{\Hol}{Hol}
\DeclareMathOperator{\Alg}{Alg}
\DeclareMathOperator{\indcoh}{indcoh}
\DeclareMathOperator{\IndCoh}{IndCoh}
\DeclareMathOperator{\IndCrys}{IndCrys}
\DeclareMathOperator{\indcrys}{indcrys}

\DeclareMathOperator{\qcoh}{qcoh}
\DeclareMathOperator{\perf}{perf}
\DeclareMathOperator{\Sh}{Sh}
\DeclareMathOperator{\QCoh}{QCoh}
\DeclareMathOperator{\QCrys}{QCrys}
\DeclareMathOperator{\crys}{crys}
\DeclareMathOperator{\LNil}{LNil}
\DeclareMathOperator{\LNilCoh}{LNilCoh}
\DeclareMathOperator{\RH}{RH}
\DeclareMathOperator{\Coh}{Coh}
\DeclareMathOperator{\Sol}{Sol}

\DeclareMathOperator{\ShHom}{\mathscr{H}\text{\kern -3pt {\calligra\large om}}\,}
\DeclareMathOperator{\sep}{sep}

\DeclareFontFamily{OT1}{pzc}{}
\DeclareFontShape{OT1}{pzc}{m}{it}{<-> s * [1.200] pzcmi7t}{}
\DeclareMathAlphabet{\mathpzc}{OT1}{pzc}{m}{it}

\renewcommand{\phi}{\varphi}

\newcommand{\factor}[2]{\left. \raise 2pt\hbox{\ensuremath{#1}} \right/
	\hskip -2pt\raise -2pt\hbox{\ensuremath{#2}}}

\makeatletter
\renewcommand\subsection{
	\renewcommand{\sfdefault}{pag}
	\@startsection{subsection}%
	{2}{0pt}{.8\baselineskip}{.4\baselineskip}{\raggedright
		\sffamily\itshape\small\bfseries
}}
\renewcommand\section{
	\renewcommand{\sfdefault}{phv}
	\@startsection{section} %
	{1}{0pt}{\baselineskip}{.8\baselineskip}{\centering
		\sffamily
		\scshape
		\bfseries
}}
\makeatother

\definecolor{gr}{rgb}{0,0.5,0}

\newcommand{\Addresses}{{
		\bigskip
		\footnotesize
		
		\textsc{\'Ecole Polytechnique F\'ed\'erale de Lausanne, SB MATH CAG, MA C3 615 (B\^atiment MA), Station 8, CH-1015 Lausanne, Switzerland}\par\nopagebreak
		\textit{E-mail address}: \texttt{jefferson.baudin@epfl.ch}

}}

\newcommand{\Acal}{\mathcal{A}}
\newcommand{\Bcal}{\mathcal{B}}

\newcommand{\Mcal}{\mathcal{M}}
\newcommand{\Ncal}{\mathcal{N}}
\newcommand{\Fcal}{\mathcal{F}}

\newcommand{\Rcal}{\mathcal{R}}
\newcommand{\Ocal}{\mathcal{O}}

\newcommand{\Ff}{\mathbb{F}}

\newcommand{\Zz}{\mathbb{Z}}

\newcommand{\Dd}{\mathbb{D}}

\author[J.~Baudin]{Jefferson Baudin} 
\date{}

\usepackage[margin=1.0in]{geometry}
\setlist{  
	listparindent=\parindent,
	parsep=0pt,
}


\title[Duality between $W_n$--Cartier crystals and perverse $\bZ/p^n\bZ$--sheaves]{Duality between $W_n$--Cartier crystals and $\bZ/p^n\bZ$--perverse sheaves}
\begin{document}
	\maketitle
	
	\begin{abstract}
		We generalize the results in \cite{Baudin_Duality_between_perverse_sheaves_and_Cartier_crystals} to obtain a duality between $W_n$--Cartier crystals and perverse $\bZ/p^n\bZ$--sheaves.
	\end{abstract}
	
	\tableofcontents
	
	\section{Introduction}
	
	\subsection{Main result}
	
	The theory of modules with Frobenius actions has generated a lot of interest and found numerous applications throughout positive characteristic commutative algebra and algebraic geometry (see, among many others, \cite{Emerton_Kisin_Riemann-Hilbert_correspondence, Bockle_Pink_Cohomological_Theory_of_crystals_over_function_fields, Bhatt_Lurie_RH_corr_pos_char, Huneke_Sharp_Bass_numbers_of_local_coh_modules, Lyubeznik_F_modules_applications_to_local_coh_and_D_mods_in_char_p, Schwede_Test_ideals_in_non_Q_Gor_rings, Blickle_Test_ideals_via_algebras_of_pe_linear_maps, Blickle_Fink_Cartier_crystals_have_finite_global_dimension, Carvajal_Rojas_Stabler_On_the_behavior_of_F_signature_uner_finite_covers, Bhatt&Co_Applications_of_perverse_sheaves_in_commutative_algebra, Hacon_Xu_On_the_3dim_MMP_in_pos_char, Patakfalvi_On_Subadditivity_of_Kodaira_dimension_in_positive_characteristic, Ejiri_Direct_images_of_pluricanonical_bundles_and_Frobenius_stable_canonical_rings_of_fibers, Hacon_Pat_GV_Geom_Theta_Divs, Baudin_Bernasconi_Kawakami_Frobenius_GR_fails}). Let us focus on the Cartier modules/crystals side. One important feature of these objects that they satisfy a duality in the form of a Riemann--Hilbert correspondence: they are equivalent to étale $\bF_p$--sheaves \cite{Schedlmeier_Cartier_crystals_and_perverse_sheaves, Baudin_Duality_between_perverse_sheaves_and_Cartier_crystals}. This means that coherent techniques can be used to the theory of perverse sheaves (e.g. the Artin vanishing theorem, see \cite[Theorem 1.1]{Bhatt&Co_Applications_of_perverse_sheaves_in_commutative_algebra} or \cite[Corollary 5.3.4]{Baudin_Duality_between_perverse_sheaves_and_Cartier_crystals}) or that topological techniques can be used to understand better the coherent picture (e.g. absolute Cohen--Macaulayness of the integral closure and Kodaira vanishing up to finite covers, see \cite[Section 5]{Bhatt&Co_Applications_of_perverse_sheaves_in_commutative_algebra}). \\
	
	Although the Cartier crystal $\omega_X$ on a smooth projective variety $X$ satisfies nice vanishing theorems up to nilpotence (\cite{Hacon_Pat_GV_Geom_Theta_Divs, Baudin_Bernasconi_Kawakami_Frobenius_GR_fails}), we recently realized that it also fails other important ones than one would hope from the characteristic zero picture, such as Grauert--Riemenschneider vanishing (see \cite[Theorem 1.1]{Baudin_Bernasconi_Kawakami_Frobenius_GR_fails}). However, it follows from \cite{Baudin_Witt_GR_vanishing_and_applications} that the collection $\{W_n\omega_X\}_{n \geq 1}$ satisfies much stronger theorems. This naturally led us to generalize our duality theorem to the setup of truncated Witt vectors to be able to deal with these objects. \\
	
	Let us fix some $n \geq 1$. The objects on the coherent side are called \emph{$W_n$--Cartier crystals}. Let us explain what this means: a (coherent) \emph{$W_n$--Cartier module} is a pair ($\cM$, $\kappa_{\cM}$), where $\cM$ is a coherent module on the scheme $W_nX \coloneqq (X, W_n\cO_X)$, and $\kappa_{\cM} \colon F_*\cM \to \cM$ is a $W_n\cO_X$--linear morphism. Note that $\kappa_{\cM}$ is in particular an endomorphism of $\cM$ as an abelian group, so it makes sense to say that $\kappa_{\cM}$ is nilpotent. The category of $W_n$--Cartier crystals is then the ``category of coherent $W_n$--Cartier modules up to nilpotence'' (see \autoref{def_crystals} for a formal definition). \\
	
	On the topological side, we show that the middle perverse t--structure of \cite{Gabber_notes_on_some_t_structures} on $D^b(X_{\et}, \bZ/p^n\bZ)$ induces a t--structure on $D^b_c(X_{\et}, \bZ/p^n\bZ)$ (i.e. we require that cohomology sheaves are constructible). The heart of this t--structure is denoted $\Perv_c(X_{\et}, \bZ/p^n\bZ)$. \\
	
	Our main result is the following:

	\begin{theorem*}[{\autoref{main_thm_final}}]
		Let $X$ be a Noetherian, $F$--finite and semi--separated $\bF_p$--scheme. Then any $W_n$--unit dualizing complex induces an equivalence of categories \[ D^b(\Crys_{W_nX}^{C})^{op} \cong D^b_c(X_{\et}, \bZ/p^n\bZ) \] commuting with derived proper pushforwards, and which sends the canonical t--structure on $D^b(\Crys_{W_nX}^C)$ to the perverse t--structure on $D^b_c(X_{\et}, \bZ/p^n\bZ)$. Hence, we obtain an equivalence  
		\[ (\Crys_{W_nX}^{C})^{op} \cong \Perv_c(X_{\et}, \bZ/p^n\bZ). \]
	\end{theorem*}

	It follows from \cite[Appendix]{Quasi_F_splittings_III} that if $X$ is of finite type over an affine Noetherian and $F$--finite $\bF_p$--algebra, then it admits a $W_n$--unit dualizing complex (see \autoref{existence_of_W_n_unit_dc}), so our assumptions are quite general. \\

	As in the case $n = 1$, the proof goes through an intermediate step: the category of $W_n$--Frobenius crystals. The link between $W_n$--Frobenius crystals and étale $\bZ/p^n\bZ$--sheaves was worked out in \cite[Section 9]{Bhatt_Lurie_RH_corr_pos_char}. Hence, our work mostly consisted of constructing a duality between $W_n$--Frobenius crystals and $W_n$--Cartier crystals, in the spirit of Grothendieck duality. The approach is identical to the case $n = 1$, and most proofs are either the same as in this case, or reduce to it. For this reason, \textbf{we strongly advise the reader to be comfortable with the techniques of \cite{Baudin_Duality_between_perverse_sheaves_and_Cartier_crystals}}. \\
	
	While developing the theory, we also generalized the fundamental finiteness theorems of Cartier modules/crystals (\cite{Blickle_Bockle_Cartier_modules_finiteness_results}) to this setup. \\
	
	Finally, we want to mention that the duality theorems developed in this paper and in \cite{Baudin_Duality_between_perverse_sheaves_and_Cartier_crystals} could probably be reproved using the language of $\infty$--categories, which could potentially simplify some technicalities (see \cite{Mattis_Weiss_The_derived_infty_category_of_Cartier_modules}).

	\subsection{Acknowledgements}\label{sec:aknowledgements}
	
	I would like to thank Fabio Bernasconi, Kay Rülling, Zsolt Patakfalvi and Nikolaos Tsakanikas for their insights. Financial support was provided by grant $\#$200020B/192035 from the Swiss National Science Foundation (FNS/SNF), and by grant $\#$804334 from the European Research Council (ERC).
	
	\subsection{Notations}
	Let $\Acal$ be an abelian category, $R$ a ring and $X$ a scheme.
	\begin{itemize}
		\item Throughout the paper, we fix a prime number $p > 0$ and integers $n, r \geq 1$. We set $q \coloneqq p^r$.
		\item The symbol $F$ always denotes the absolute Frobenius.
		
		\item The category of $\Ocal_X$-modules (resp. $R$-modules) is denoted by $\Mod(\Ocal_X)$ (resp. $\Mod(R)$).
		
		\item The category of quasi-coherent (resp. coherent) sheaves on $X$ is denoted by $\QCoh_X$ (resp. $\Coh_X$).
		
		\item Given any ring $\Lambda$, the category of $\Lambda$-modules on the \'etale site is denoted $\Sh(X_{\et}, \Lambda)$, and its derived category is $D(X_{\et}, \Lambda)$. 
		
		\item Given a point $x \in X$, we write $\overline{x} \coloneqq \Spec k(x)^{\sep}$, where $k(x)^{\sep}$ denotes a separable closure of $k(x)$. 
		
		\item The symbol $D(\Acal)$ denotes the derived category of $\cA$. Given $G \colon \cA \to \cB$ a left exact functor, its right derived functor is denoted $RG$ (when it exists). If $G$ is already exact, we will simply write $G = RG$.
		
		\item For a cochain complex $A^\bullet = (A^n)_{n \in \Zz}$ in $\Acal$, the $i$'th cohomology object of this cochain is denoted by $\cH^i(A^\bullet)$. 
		
		\item Given a weak Serre subcategory $\Bcal \inc \Acal$ (see \cite[{\href{https://stacks.math.columbia.edu/tag/02MO}{Tag 02MO}}]{Stacks_Project}), the triangulated category $D_\Bcal(\Acal)$ denotes the full subcategory of $D(\Acal)$ of complexes $A^\bullet$ such that $\cH^i(A^\bullet) \in \Bcal$ for all $i \in \Zz$ (see \cite[{\href{https://stacks.math.columbia.edu/tag/06UQ}{Tag 06UQ}}]{Stacks_Project}). 
		
		\item Given an abelian category $\cA$ and a Serre subcategory $\cB$ (see \cite[{\href{https://stacks.math.columbia.edu/tag/02MO}{Tag 02MO}}]{Stacks_Project}), the quotient will be denoted $\cA/\cB$ (see \cite[{\href{https://stacks.math.columbia.edu/tag/02MS}{Tag 02MS}}]{Stacks_Project}).
		
		\item We loosely denote $D_{coh}(\QCoh_X) \coloneqq D_{\Coh_X}(\QCoh_X)$. Throughout the paper, we will make comparable abuse of notations.
		
		\item The full subcategories $D^+(\Acal)$, $D^-(\Acal)$ $D^b(\Acal)$ denote complexes $A^\bullet$ such that $\cH^i(A^\bullet) = 0$ for $i \ll 0$ (resp. $i \gg 0$ and $\abs{i} \gg 0$). We can combine this notation with the one above.
		
		\item Given a finite morphism $f \colon X \to Y$ of Noetherian schemes, the right adjoint of $f_*$ (see \stacksproj{08YP}) will be denoted $f^{\flat}$, and we shall write $f^! \coloneqq Rf^{\flat}$. For example, if $f$ is a closed immersion (with associated ideal $\cI \inc \cO_Y$), then $f^{\flat}(\cM)$ is nothing else than the $\cI$--torsion elements in $\cM$. 
		
		\item More generally, given a morphism of finite type $f \colon X \to Y$ of Noetherian schemes where $Y$ admits a dualizing complex, the upper--schriek functor $f^!$ is the functor defined in \stacksproj{0AU5} (in \emph{loc. cit.}, it is named $f_{new}^!$).
		
		\item Given an abelian group $A$, we denote by $A[p]$ its subgroup of $p$--torsion elements. We use the same notation for sheaves.
		
	\end{itemize}
	
	\section{Witt vector preliminaries}\label{section preliminaries}
	
	Here we gather general results about Witt vectors that will be used throughout the paper.
	
	\begin{snotation}
		\begin{itemize}
			\item Given an $\bF_p$--algebra $R$, the ring of $n$--truncated Witt vectors (see \cite[Section 9.3]{Gabber_Ramero_Almost_ring_theory}) is denoted $W_n(R)$. The \emph{Frobenius} $F \colon W_n(R) \to W_n(R)$ is the ring morphism $(a_1, \dots, a_n) \mapsto (a_1^p, \dots, a_n^p)$, the \emph{Verschiebung} $V \colon F_*W_n(R) \to W_{n + 1}(R)$ is the $W_{n + 1}(R)$--linear map $(a_1, \dots, a_n) \mapsto (0, a_1, \dots, a_n)$, and the \emph{restriction} map $R \colon W_{n + 1}(R) \to W_n(R)$ is the ring morphism $(a_1, \dots, a_{n +1}) \mapsto (a_1, \dots, a_n)$. Given $s \in R$, its Teichmüller lift is the element $[s] \coloneqq (s, 0, \dots, 0) \in W_n(R)$.
			\item Given an $\bF_p$--scheme $X$, associating to each open $U \inc X$ the ring $W_n\cO_X(U)$ gives a sheaf of rings $W_n\cO_X$ on $X$. The associated ringed space $(X, W_n\cO_X)$ is a $\bZ/p^n\bZ$--scheme, and is denoted $W_nX$.
			\item Given an $\bF_p$--scheme $X$ and a quasi--coherent ideal $\cI_0 \inc \cO_X$, we denote by $W_n\cI_0 \inc W_n\cO_X$ the quasi--coherent ideal of elements $(s_1, \dots s_n)$ such that each $s_i \in \cI_0$.
			\item Given a morphism $f \colon X \to Y$ of $\bF_p$--schemes, the induced morphism $W_nX \to W_nY$ will still be denoted $f$.
		\end{itemize}
	\end{snotation}

	The following two results will be implicitly used throughout the article without further mention.

	\begin{slem}\label{general_statements_Witt_vectors}
		Let $f \colon X \to Y$ be a morphism of Noetherian, $F$--finite $\bF_p$--schemes satisfying a property $(P)$, where $(P)$ is one of the following:
		\begin{itemize}
			\item of finite type;
			\item proper;
			\item separated;
			\item semi--separated;
			\item finite;
			\item affine;
			\item closed immersion;
			\item open immersion;
			\item étale.
		\end{itemize}
	Then the induced map $W_nX \to W_nY$ also satisfies $(P)$.
	\end{slem}
	\begin{proof}
		The proof for affine, semi--separated, open and closed immersions are immediate. All the other cases except for finiteness are included in \cite[Appendix A.1]{Langer_Zink_De_Rham_Witt_cohomology_for_a_proper_and_smooth_morphism}. The finite case follows from the proper and affine cases.
	\end{proof}

	\begin{slem}
		Let $X$ be a Noetherian $F$--finite $\bF_p$--scheme. Then $W_nX$ is also Noetherian, and the Frobenius map $F \colon W_nX \to W_nX$ is also finite.
	\end{slem}
	\begin{proof}
		Noetherianness follows from \cite[Proposition A.4]{Langer_Zink_De_Rham_Witt_cohomology_for_a_proper_and_smooth_morphism}, and finiteness of the Frobenius is a consequence of \autoref{general_statements_Witt_vectors}.
	\end{proof}

	\begin{slem}\label{Witt_ideal_inclusions}
		Let $X$ be a Noetherian $F$--finite scheme, and let $\cI_0 \inc \cO_X$ be a quasi--coherent ideal. Then for any $m \geq 1$, there exists $e \geq 1$ such that \[ F^e(W_n\cI_0) \inc (W_n\cI_0)^m. \] In particular, we have that $W_n(\cI_0^{m'}) \subseteq (W_n\cI_0)^m$ for $m' \gg 0$.
	\end{slem}
	\begin{proof}
		We will show the result by induction on $n \geq 1$. If $n = 1$, the result is immediate by Noetherianness.
		
		For $n \geq 2$, let $\cM$ denote the kernel of $R^{n - 1} \colon (W_n\cI_0)^m \to \cI_0^m$. Then we have inclusions $F_*(W_{n - 1}\cI_0)^m \inc \cM \inc F_*W_{n - 1}\cI_0$ via the Verschiebung. In particular, the cokernel of the injection $\cM \inj F_*W_{n - 1}\cI_0$ is killed by some power of the Frobenius by induction. Now, consider the following commutative diagram of exact sequences
		\[ \begin{tikzcd}
			0 \ar[rr] & & \cM \ar[rr, "V"] \ar[d, "\inc"] & & (W_n\cI_0)^m \ar[rr, "R^{n - 1}"] \ar[d, "\inc"] && \cI_0^m \ar[d, "\inc"] \ar[rr] && 0 \\
			0 \ar[rr] & & F_*W_{n - 1}\cI_0 \ar[rr, "V"] & & W_n\cI_0 \ar[rr, "R^{n - 1}"] && \cI_0 \ar[rr]&& 0.
		\end{tikzcd} \]
		
		Since both the left and right vertical arrows have their cokernel killed by some power of the Frobenius, this holds for the middle arrow too.
	\end{proof}

	\begin{slem}\label{base_change_etale_Witt}
		Let $f \colon U \to X$ be an étale morphism of Noetherian $F$--finite $\bF_p$--algebras. Then the diagram
		\[ \begin{tikzcd}
			W_mU \arrow[r] \arrow[d] & W_nU \arrow[d] \\
			W_mX \arrow[r, hook]     & W_nX          
		\end{tikzcd} \] is Cartesian.
	\end{slem}
	\begin{proof}
		We may assume that both $U$ and $X$ are affine. It then follows from \cite[Proposition A.14 and Corollary A.18]{Langer_Zink_De_Rham_Witt_cohomology_for_a_proper_and_smooth_morphism}.
	\end{proof}
	
	\section{The theory of Witt--Frobenius modules}
	
	Here we gather fundamental results on Witt--Frobenius modules, such as their Riemann--Hilbert correspondence (see \cite{Bhatt_Lurie_RH_corr_pos_char}).
	
	\subsection{Basic definitions}
	\begin{defn}\label{def F-module}
		Let $X$ be an $\Ff_p$-scheme, and let $\Mcal$ be an $W_n\Ocal_X$-module over $X$.
		\begin{itemize}
			\item An \emph{$r$--Frobenius module over $W_nX$} (or $W_n$--Frobenius module when the context is clear) structure on $\Mcal$ is a morphism of $W_n\Ocal_X$-modules $\tau_{\cM} \colon \Mcal \ra F^r_*\Mcal$. A morphism in this category is a morphism $\theta \colon \Mcal \ra \Ncal$ between two $W_n\Ocal_X$-modules such that the square \[ \begin{tikzcd}
				\Mcal \arrow[r, "\theta"] \arrow[d, "\tau_{\cM}"'] & \Ncal \arrow[d, "\tau_{\cN}"] \\
				F^r_*\Mcal \arrow[r, "F^r_*\theta"'] & F^r_*\Ncal			
			\end{tikzcd} \] commutes. The category of $W_n$--Frobenius modules is denoted $\Mod(W_n\Ocal_X[F^r])$.
			\item The full subcategory of Frobenius modules which are quasi-coherent (resp. coherent) as $W_n\cO_X$-modules is denoted by $\QCoh_{W_nX}^{F^r}$ (resp. $\Coh_{W_nX}^{F^r}$).
			\item By adjunction, the map $\tau_{\cM} \colon \cM \to F^r_*\cM$ is equivalent to a map $\tau_{\cM}^* \colon F^{r, *}\cM \to \cM$. This morphism is called the \emph{adjoint structural morphism}. A $W_n$--Frobenius module is said to be \emph{unit} if its adjoint structural morphism is an isomorphism.
		\end{itemize}
	\end{defn}
	\begin{rem}\label{operations_on_F_modules}
		\begin{itemize}
			\item As for $n = 1$, one can take pushforwards, pullbacks and tensor products of $W_n$--Frobenius modules.
			\item The reason why the category of $W_n$--Frobenius modules is called $\Mod(W_n\cO_X[F^r])$ is the same as in the case where $n = 1$ (see \cite[Construction after 2.2.3]{Baudin_Duality_between_perverse_sheaves_and_Cartier_crystals}).
		\end{itemize}

	\end{rem}

	\begin{defn}
		Let $X$ be a Noetherian $F$--finite $\bF_p$--scheme, and let $\cM \in \QCoh_{W_nX}^{F^r}$.
		\begin{itemize}
			\item We say that $\cM$ is \emph{ind-coherent} if it is the union of its coherent $W_n$--Frobenius submodules. The full category of these objects is denoted by $\IndCoh_{W_nX}^{F^r}$.
			\item We say that $\cM$ is \emph{nilpotent} if $\tau_{\cM}^n = 0$ for some $n \geq 0$, where we make the abuse of notations \[ \tau_{\cM}^n \coloneqq  F^{r(n - 1)}_*\tau_{\cM} \circ \dots \circ F^r_*\tau_{\cM} \circ \tau_{\cM}. \]
			\item It is said to be \emph{locally nilpotent} it is the union of its nilpotent $W_n$--Frobenius submodules. The full subcategory of locally nilpotent modules is called $\LNil$ (note that $\LNil \inc \IndCoh_{W_nX}^{F^r}$).
			\item We set $\Nil \coloneqq \LNil \cap \Coh_{W_nX}^{F^r}$.
			\item We define $\QCrys_X^{F^r} \coloneqq \QCoh_X^{F^r}/\LNil$, $\IndCrys_X^{F^r} \coloneqq \IndCoh_X^{F^r}/\LNil$, and the category of \emph{$W_n$--Frobenius crystals} to be $\Crys_X^{F^r} \coloneqq \Coh_X^{F^r}/\Nil$.
		\end{itemize}
	\end{defn}
	
	\begin{lem}\label{all cats of F-mods are Grothendieck}
		Let $X$ be a Noetherian $F$--finite $\bF_p$-scheme, and let $\cC$ be one of the following categories:  \[\Mod(W_n\cO_X), \QCoh_{W_nX}, \Mod(W_n\cO_X[F^r]), \QCoh_{W_nX}^{F^r}, \IndCoh_{W_nX}^{F^r}, \QCrys_{W_nX}^{F^r}, \IndCrys_{W_nX}^{F^r}.\]
		Then $\cC$ is a Grothendieck category. In particular, it has enough injectives.
	\end{lem}
	\begin{proof}
		The proof is identical to that of \cite[Lemma 2.2.10]{Baudin_Duality_between_perverse_sheaves_and_Cartier_crystals}.
	\end{proof}
	
	\begin{lem}\label{all derived pushforwards are the same for F-modules}
		Let $f \colon X \to Y$ be a proper morphism between semi--separated, Noetherian and $F$--finite $\bF_p$-schemes. Then the diagram 
		\[ \begin{tikzcd}
			D^+(\IndCoh_{W_nX}^{F^r}) \arrow[d] \arrow[rr, "Rf_*"] &  & D^+(\IndCoh_{W_nY}^{F^r}) \arrow[d] \\
			D^+(\cC_X) \arrow[rr, "Rf_*"]                     &  & D^+(\cC_Y)                 
		\end{tikzcd} \] commutes for any \[ \cC \in \left\{\Mod(W_n\cO), \Mod(W_n\cO[F^r]), \QCoh_{W_n}^{F^r}, \QCoh_{W_n}, \IndCrys_{W_n}^{F^r}\right\}.\]
	\end{lem}
	
	\begin{proof} 
		The proof is identical to the case $n = 1$, see \cite[Proposition 2.2.11 and Remark 2.2.12]{Baudin_Duality_between_perverse_sheaves_and_Cartier_crystals} (note that the first paragraph of the proof in \emph{loc. cit.} can be disregarded here, since we added a semi--separatedness assumption).
	\end{proof}

	For now on, fix a Noetherian $F$--finite $\bF_p$--scheme $X$. Note that if $1 \leq m \leq n$, then any $W_m$--Frobenius module is also a $W_n$--Frobenius module. Nilpotence is certainly preserved, so by \stacksproj{06XK} we obtain a morphism at the level of crystals.
	
	\begin{lem}\label{Crys_supported_at_p^m_is_W_m}
		Let $1 \leq m \leq n$ be an integer. Then the functors $\Crys_{W_mX}^{F^r} \to \Crys_{W_nX}^{F^r}$ and $\IndCrys_{W_mX}^{F^r} \to \IndCrys_{W_nX}^{F^r}$ explained above are fully faithful, and their essential image consist of objects $\cM$ such that $p^m\cM \sim_F 0$.
	\end{lem}
	\begin{proof}
		Since this functor preserves colimits, it is enough to show the statement for $\Crys_{W_nX}^{F^r}$. Faithfulness follows from \stacksproj{06XK}. Now, let $\cM, \cN \in \Crys_{W_mX}^{F^r}$, and let $f \colon \cM_1 \dra \cM_2$ be a morphism in $\Crys_{W_nX}^{F^r}$. By the exact same proof as in \cite[Proposition 3.4.6]{Bockle_Pink_Cohomological_Theory_of_crystals_over_function_fields}, there exists a commutative diagram
		\[ \begin{tikzcd}
			\cM \arrow[rr, "f", dashed] \arrow[rrd, "g"'] &  & \cN \arrow[d, "\tau^e_{\cN}"] \\
			&  & F^{er}_*\cN                  
		\end{tikzcd} \] in $\Crys^{F^r}_{W_nX}$, where $g$ is a morphism of $W_m$--Frobenius modules. Since $\tau_{\cN}$ is a nil-isomorphism, we conclude that $f$ is in the image of $\Hom_{\Crys^{F^r}_{W_mX}}(\cM, \cN)$, so we also obtained fullness.
		
		The essential image of our functor is certainly contained in $p^m$-torsion crystals. Conversely, let $\cM \in \Crys^{F^r}_{W_nX}$ satisfying $p^m\cM \sim_F 0$. Notice that the kernel of the canonical surjection $\overline{R^{n - m}} \colon W_n\cO_X/p^m \to W_m\cO_X$ is killed by some power of the Frobenius, say $F^e (\ker(\overline{R^{n - m}})) = 0$. Then \[ \ker(\overline{R^{n - m}}) \cdot F^e_*\cN = 0, \] so $F^e_*\cN$ defines a $W_m$--Cartier module. Since $\tau_{\cM}^e$ is a nil--isomorphism, we have proven that $\cM$ is in the essential immage of $\Crys^{F^r}_{W_mX} \to \Crys_{W_nX}^{F^r}$.
	\end{proof}
	
	\subsection{The Riemann--Hilbert correspondence}
	\begin{defn-constr}[{\cite[Construction 9.3.1]{Bhatt_Lurie_RH_corr_pos_char}}]\label{extending F-modules to the etale site}
		Let $X$ be a Noetherian $F$--finite 
		$\bF_q$--scheme, and let $\cM \in \QCoh_{W_nX}^{F^r}$. We can extend both $\Mcal$ and $\tau_{\cM}$ to the étale site, and obtain a morphism $\tau_{{\cM}_{\et}} \colon \Mcal_{\et} \ra F^r_*\Mcal_{\et}$ (recall that we can pullback $W_n$--Frobenius modules).
	
		Define $\Sol_n(\Mcal) \coloneqq \ker(\tau_{{\cM}_{\et}} - 1)$, taken in the abelian category $\Sh(X_{\et}, W_n(\Ff_q))$ of \'etale $W_n(\Ff_q)$--sheaves on $X$. This construction defines a functor $\Sol_n \colon \QCoh_{W_nX}^{F^r} \ra \Sh(X_{\et}, W_n(\Ff_q))$.
	\end{defn-constr}
	\begin{rem}
		Since the inclusion $X \inj W_nX$ is a universal homeomorphism by \stacksproj{0CNF}, we identify the étale sites of $X$ and $W_nX$ by \stacksproj{03SI}.
	\end{rem}
	
	\begin{thm}[{\cite[Theorem 9.6.1]{Bhatt_Lurie_RH_corr_pos_char}}]\label{main thm Bhatt Lurie}
		Let $X$ be a Noetherian, $F$--finite $\bF_q$--scheme. Then the functor $\Sol_n$ vanishes on $\LNil$, and induces equivalences of categories 
		\[ \begin{tikzcd}[row sep=small]
			\IndCrys_{W_nX}^{F^r} \arrow[rr, "\cong"] &  & {\Sh(X_{\et}, W_n(\bF_q))}, \mbox{ and }  \\
			\Crys_{W_nX}^{F^r} \arrow[rr, "\cong"]    &  & {\Sh^c(X_{\et}, W_n(\bF_q))},
		\end{tikzcd} \] where $\Sh^c(X_{\et}, W_n(\bF_q))$ denotes the category of constructible étale $W_n(\bF_q)$--sheaves.
		
		This functor preserves tensor products, pullbacks and derived proper pushforwards (there is no ambiguity in what this means by \autoref{all derived pushforwards are the same for F-modules}). Furthermore, given $1 \leq m \leq n$, the diagram \[  \begin{tikzcd}
			\IndCrys_{W_mX}^{F^r} \arrow[rr, hook] \arrow[d, "\Sol_m"'] &  & \IndCrys_{W_nX}^{F^r} \arrow[d, "\Sol_n"] \\
			{\Sh(X_{\et}, W_m(\bF_q))} \arrow[rr, hook]                 &  & {\Sh(X_{\et}, W_n(\bF_q))}               
		\end{tikzcd} \] commutes.
	\end{thm}
	For this proof, we will need a few definitions from \cite{Bhatt_Lurie_RH_corr_pos_char}.
	
	\begin{defn}
		Let $X$ be a Noetherian $F$--finite $\bF_p$--scheme, and let $\cM \in \QCoh_{W_nX}^{F^r}$.
		\begin{itemize}
			\item The \emph{perfection} of $\cM$ is by definition \[ \cM^{1/p^{\infty}} \coloneqq \colim F^{er}_*\cM. \]
			We say that $\cM$ is \emph{perfect} if the natural map $\cM \to \cM^{1/p^{\infty}}$ is an isomorphism (i.e. if $\tau_\cM$ is an isomorphism).
			\item We say that $\cM$ is \emph{holonomic} if it the perfection of a coherent $W_n$--Frobenius module. The full subcategory of holonomic $W_n$--Frobenius modules is denoted $\Hol_{W_nX}^{F^r}$.
			\item We say that $\cM$ is \emph{algebraic} if it is the union of its holonomic $W_n$--Frobenius submodules (in particular it must be perfect). The full subcategory of algebraic $W_n$--Frobenius modules is denoted $\Alg_{W_nX}^{F^r}$.
		\end{itemize} 
	\end{defn}

	\begin{rem}\label{definitions_of_holonomicity}
		When $X$ is affine, our definition of algebraic $W_n$--Frobenius modules agrees with \cite[Definition 9.5.2]{Bhatt_Lurie_RH_corr_pos_char}. Indeed, condition (c) of Proposition 9.5.1 in \emph{loc. cit.} is exactly the same as saying that any element lives in a finitely generated $W_n$--Frobenius module (and hence a holonomic one too by taking perfection).
		
		Note that being holonomic and algebraic can be checked locally (this follows by the same argument as in \cite[Section 10.4]{Bhatt_Lurie_RH_corr_pos_char}).
	\end{rem}

	\begin{lem}\label{perfection_induces_iso_crys_and_hol}
		Let $X$ be a Noetherian, $F$--finite $\bF_p$--scheme. Then perfection induces equivalences of categories \[ \Crys_{W_nX}^{F^r} \to \Hol_{W_nX}^{F^r} \] and \[ \IndCrys_{W_nX}^{F^r} \to \Alg_{W_nX}^{F^r}. \]
	\end{lem}
	\begin{proof}
		Since perfection preserves colimits, it is enough to show the first statement. By definition, perfection gives an essentially surjective exact functor $\Coh_{W_nX}^{F^r} \to \Hol_{W_nX}^{F^r}$. Since it exactly annihilates nilpotent objects, we know by \stacksproj{06XK} that it induces an essentially surjective faithful functor $\Crys_{W_nX}^{F^r} \to \Hol_{W_nX}^{F^r}$. It remains to show that this functor is full, so let $f \colon \cM^{1/p^{\infty}} \to \cN^{1/p^{\infty}}$ be a morphism of holonomic $W_n$--Frobenius modules, where $\cM$ and $\cN$ are coherent. Since $\cM$ is a Noetherian object, the composition $\cM \to \cM^{1/p^\infty} \to \cN^{1/p^{\infty}}$ factors through $F^{er}_*\cN$ for some $e > 0$. Say that $g \colon M \to F^{er}_*\cN$ is the induced map, and let $h \coloneqq (\tau_N^{-1})^e \circ g \colon \cM \dashrightarrow \cN$ (it is only a morphism of crystals!). Then by construction, $h^{1/p^{\infty}} = f$, so we have proven fullness too.
	\end{proof}
	
	\begin{proof}[Proof of \autoref{main thm Bhatt Lurie}]
		We start with the first equivalence. As in \cite[Theorem 10.2.7]{Bhatt_Lurie_RH_corr_pos_char}, it is enough to show this when $X$ is affine, say $X = \Spec(R)$. This is then a consequence of \autoref{perfection_induces_iso_crys_and_hol} and \cite[Theorem 9.6.1]{Bhatt_Lurie_RH_corr_pos_char} (note that the solution functor is invariant under taking perfection by the same proof as in \cite[Proposition 3.2.9]{Bhatt_Lurie_RH_corr_pos_char}). The second equivalence follows directly from the first one, since an equivalence of categories preserves compact objects. 
		
		The proof that $\Sol_n$ preserves pullbacks is the same as that in the case $n = 1$ (see e.g. \cite[Proposition 10.2.2]{Bockle_Pink_Cohomological_Theory_of_crystals_over_function_fields}). To show that it preserves derived pushforwards, it suffices to show as in the proof of \cite[Theorem 10.5.5]{Bhatt_Lurie_RH_corr_pos_char} that given a proper morphism $f \colon X \to Y$ and $\cM \in \Alg_{W_nX}^{F^r}$, we have that $R^if_*\cM \in \Alg_{W_nY}^{F^r}$ for all $i \geq 0$. Since the functors $R^if_*$ preserve coherence and perfection, they preserve holonomicity. Hence, they also preserve algebraicity. 
		
		Let us to deal with tensor products.  By definition of $\Sol_n$ and the Frobenius structure on the tensor product (apply the structural map on each element in the tensor), it is immediate that there is a natural transformation \[ \Sol_n(-) \otimes_{W_n{\bF_q}} \Sol_n(-) \to \Sol_n(- \otimes_{W_n\cO_X} -). \] Since both $\Sol$ and tensor products preserve colimits, it is enough to show the result on $\Crys_{W_nX}^{F^r}$. Since we can check it on stalks, we may assume that $X = \Spec k$ with $k$ algebraically closed. It is enough to show that $\RH_n \coloneqq \Sol_n^{-1}$ preserves tensor products. However, in this case it is immediate to see that (up to natural isomorphism), $\RH_n$ sends a finite $W_n(\bF_q)$--module $V$ to $(V \otimes_{W_n(\bF_q)} W_n(k)$, $1 \otimes F^r)$, so compatibility with tensor product follows.
		
		Finally, let us show the last compatibilities, so fix $1 \leq m \leq n$. Since for any étale morphism $f \colon U \to X$, the diagram
		 \[ \begin{tikzcd}
			W_mU \arrow[r] \arrow[d, hook] & W_mX \arrow[d, hook] \\
			W_nU \arrow[r]     & W_nX          
		\end{tikzcd} \] is a base change square by \autoref{base_change_etale_Witt}, we deduce by flat base change that the functor $(\cdot)_{\et}$ of \autoref{extending F-modules to the etale site} commutes with $\QCoh_{W_mX}^{F^r} \inj \QCoh_{W_nX}^{F^r}$. Hence, we are done.
	\end{proof}

	\begin{rem}
		Thanks to the last compatibility in \autoref{main thm Bhatt Lurie}, we will remove the subscript $n$ in the notation $\Sol_n$.
	\end{rem}

	\subsection{Comparison between certain derived categories}

	Fix a Noetherian $F$--finite $\bF_p$--scheme $X$. Besides the Riemann--Hilbert correspondence, an important tool that we will need for the duality is to show that the natural functor $D^b(\Coh_{W_nX}^{F^r}) \to D^b_{\coh}(\QCoh_{W_nX}^{F^r})$ is an equivalence of categories. We will also need it at the level of crystals, and some variations of this statement.
	
	\begin{defn}
		Let $\cM \in \QCoh_{W_nX}^{F^r}$. We denote by $\ind_{W_nX}(\cM) \subseteq \cM$ the biggest ind--coherent $W_n$--Frobenius submodule of $\cM$.
	\end{defn}

	The main technical ingredient to show what we want is to show that given $\cM \in \IndCoh_{W_nX}^{F^r}$, we have that $R^i\ind(\cM) = 0$ for all $i > 0$. Unlike the case of $W_n$--Cartier modules, we could not find a way to reduce to the case $n = 1$. 
	
	\begin{prop}\label{ind_coh_is_ind_acyclic_F_modules}
		Let $\cM \in \IndCoh_{W_nX}^{F^r}$. Then for all $i > 0$, $R^i\ind_{W_nX}(\cM) = 0$.
	\end{prop}
	\begin{proof}
		The same proof as in \cite[Theorem 5.2.10]{Bockle_Pink_Cohomological_Theory_of_crystals_over_function_fields} works. The only few ``differences'' are in gluing the localisation functor in Proposition 5.1.13 and Lemma 5.1.12, and in the proofs of Propositions 4.5.1 and 4.5.2.
		
		For the last two statements in \emph{loc. cit.}, they localized by any element of their ring. In the case of Witt vectors, we only need to localize by Teichmüller lifts though, since $(W_nR)_{[t]} = W_n(R_t)$ for any $t \in R$. Thus, although the proof of Lemma 5.1.12 would not work if we localized by any $t' \in W_nR$ (due to the difference between $F^r(t')$ and $(t')^q$), it works perfectly fine for Teichmüller lifts. This is the right context in which 5.1.12 and 5.1.13 work, so we can indeed glue their localisation functor.
		
		Now, let us mention 4.5.1 and 4.5.2. In their proof, there is an ideal $\cI_0 \subseteq \cO_X$ that comes up, and an important property is that $F^r(\cI) \subseteq \cI^q$. The analogue that shows up here is an ideal of the form $W_n\cI_0$, but it seems to be wrong in general that $F^r(W_n\cI_0) \subseteq (W_n\cI_0)^q$. The way around is to use the ideals $W_n(\cI_0)$, so that $F^r(W_n\cI_0) \subseteq W_n(\cI_0^q)$. This can indeed be done thanks to \autoref{Witt_ideal_inclusions}. For example, following the proof of 4.5.1, we have that $(W_n\cI_0)^m\tau(F^{r, *}\ttilde{\cF}) \inc \ttilde{\cF}$ for some $m > 0$. We can then use \autoref{Witt_ideal_inclusions} to deduce that also $(W_n\cI_0^{m'})\tau(F^{r, *}\ttilde{\cF}) \inc \ttilde{\cF}$ for some (perhaps bigger) $m' > 0$.
	\end{proof}

	\begin{prop}\label{D^b(Coh) = D^b_{coh}(QCoh) for F-modules}
		Let $X$ be a Noetherian $F$--finite $\Ff_p$-scheme. Then for $* \in \{-, b\}$, the natural functors 
		\[ \begin{tikzcd}[row sep=small]
			D^*(\Coh_{W_nX}^{F^r}) \arrow[rr]  &  & D^*_{\coh}(\QCoh_{W_nX}^{F^r}), \mbox{ and }  \\
			D^*(\Crys_{W_nX}^{F^r}) \arrow[rr] &  & D^*_{\crys}(\QCrys_{W_nX}^{F^r})
		\end{tikzcd}  \] 
		are equivalences of categories, where by $D^*_{\crys}(\QCrys_{W_nX}^{F^r})$ we mean complexes whose cohomology sheaves are in the essential image of $\Crys_{W_nX}^{F^r} \to \QCrys_{W_nX}^{F^r}$.
	\end{prop}
	\begin{proof}
		Thanks to \autoref{ind_coh_is_ind_acyclic_F_modules}, the proof is now identical to that of \cite[Theorem 5.3.1]{Bockle_Pink_Cohomological_Theory_of_crystals_over_function_fields}.
	\end{proof}

	\begin{lem}\label{D^b(Coh) = D^b_{coh}(O_X) for F-modules}
		Let $X$ be a semi--separated Noetherian $F$-finite $\bF_p$--scheme. Then the natural functor \[ D^b(\Coh_{W_nX}^{F^r}) \to D^b_{\coh}(W_n\cO_X[F^r]) \] is an equivalence.
	\end{lem}
	\begin{proof}
		The proof is identical to that of \cite[Proposition 2.2.15]{Baudin_Duality_between_perverse_sheaves_and_Cartier_crystals}.
	\end{proof}
	
	\section{The theory of Witt--Cartier modules}
	
	Here we develop the theory of $W_n$--Cartier modules, and prove important properties that they inherit from usual Cartier modules. Throughout this section, we fix a Noetherian $F$--finite $\bF_p$--scheme $X$.
	
	\subsection{Basic definitions}
	\begin{defn}\label{def Cart module}
		\begin{itemize}
			\item A \emph{$W_n$--$r$--Cartier module} (or simply \emph{$W_n$--Cartier module}) is a pair $(\cM, \kappa_{\cM}$) where $\Mcal$ is a $W_n\Ocal_X$-module, and $\kappa_{\cM} \colon F^r_*\cM \to \cM$ is a $W_n\cO_X$--linear morphism, called the \emph{structural morphism} of $\cM$. A morphism of $W_n$--Cartier modules is a morphism of underlying $W_n\cO_X$--modules commuting with the respective Cartier module structures. We denote the category of Cartier modules by $\Mod(W_n\Ocal_X[F^r]^{op})$ or $\Mod(W_n\Ocal_X[C^r])$. 
			\item The full subcategory of quasi-coherent (resp. coherent, ind--coherent) Cartier modules is denoted by $\QCoh_{W_nX}^{C^r}$ (resp. $\Coh_{W_nX}^{C^r}$, $\IndCoh_{W_nX}^{C^r}$).
			\item Since $F \colon W_nX \to W_nX$ is finite, given any quasi--coherent $W_n$--Cartier module $(\cM, \kappa_{\cM})$, there is an adjoint morphism $\kappa_{\cM}^{\flat} \colon \cM \to F^{r, \flat}\cM$. This morphism is called the \emph{adjoint structural morphism}.
			\item A quasi--coherent $W_n$--Cartier module is said to be \emph{unit} if its adjoint structural morphism is an isomorphism. The full subcategory of unit (resp. unit and ind--coherent) $W_n$--Cartier module is denoted $\QCoh_{W_nX}^{C^r, \unit}$ (resp. $\IndCoh_{W_nX}^{C^r, \unit}$).
			\item As in the case $n = 1$, although the forgetful functors $\QCoh_{W_nX}^{C^r, \unit} \to \QCoh_{W_nX}^{C^r}$ and $\IndCoh_{W_nX}^{C^r, \unit} \to \IndCoh_{W_nX}^{C^r}$ are left--exact, they may not be exact. The issue is that the cokernel of a morphism of unit $W_n$--Cartier modules may not be unit, since $F^{r, \flat}$ is only left--exact.
		\end{itemize}
	\end{defn}
	
	\begin{rem}\label{rem def Cartier module}
		\begin{enumerate}
			\item Unless stated otherwise, the structural morphism of a $W_n$--Cartier module $\cM$ will always be denoted $\kappa_{\cM}$, and its adjoint structural morphism will always be denoted $\kappa_{\cM}^{\flat}$.
			\item As in the case $n = 1$, we can take pushforwards of $W_n$--Cartier modules. When we work with quasi--coherent ones, we can also apply the functor $f^{\flat}$ for a finite morphism $f$, and this gives a right adjoint to the pushforward.
		\end{enumerate}
	\end{rem}
	\begin{lem}\label{all cats of Cartier mods are Grothendieck}
		Let $\cC$ be one of the following categories:  \begin{equation*}
			\Mod(W_n\cO_X[C^r]), \QCoh_{W_nX}^{C^r}, \IndCoh_{W_nX}^{C^r}, \QCoh_{W_nX}^{C^r, \unit}, \IndCoh_{W_nX}^{C^r, \unit}.
		\end{equation*}
		Then $\cC$ is a Grothendieck category. In particular, it has enough injectives.
	\end{lem}
	\begin{proof}
		The proof is identical to that of \cite[Lemma 2.2.10]{Baudin_Duality_between_perverse_sheaves_and_Cartier_crystals}.
	\end{proof}
	
	\begin{lem}\label{D(QCoh) = D_qcoh(O_X[C])}
		For $* \in \{+, b\}$, the natural functor $D^*(\QCoh_{W_nX}^{C^r}) \ra D^*_{\qcoh}(W_n\Ocal_X[C^r])$ is an equivalence of categories.
	\end{lem}
	\begin{proof}
		The proof is identical to that of \cite[Proposition 3.1.7]{Baudin_Duality_between_perverse_sheaves_and_Cartier_crystals}.
	\end{proof}

	\subsection{Ind-coherent Cartier modules}
	In the duality proven in \cite{Baudin_Duality_between_perverse_sheaves_and_Cartier_crystals}, two crucial facts about ind--coherent Cartier modules were needed:
	\begin{itemize}
		\item an injective object in $\IndCoh_X^{C^r}$ is also injective in $\Mod(\cO_X)$;
		\item the natural functor $D^b(\Coh_X^{C^r}) \to D^b_{\coh}(\QCoh_X^{C^r})$ is an equivalence of categories.
	\end{itemize}
	
	The goal here is to prove the $W_n$--analogue of these results. The proof of both facts will be by restricting to the case $n = 1$.

	\begin{lem}\label{ind_coh_stable_under_everything}
		Ind--coherence is stable by sub--objects, quotients, extension and colimits.
	\end{lem}
	\begin{proof}
		Stability by sub--objects, quotients and colimits are immediate from the definitions. To obtain stability under extensions, apply the same proof as in \cite[Proposition 2.0.8]{Blickle_Bockle_Cartier_crystals}.
	\end{proof}

	Let $\cM$ be a $W_n$--Cartier module. Note that since $F(p) = p$, both $p\cM$ and $\cM[p]$ (i.e. the $p$--torsion submodule of $\cM$) are $W_n$--Cartier submodule of $\cM$.

	\begin{lem}[{\cite[Proposition 9.5.1]{Bhatt_Lurie_RH_corr_pos_char}}]\label{indcoh_iff_indcoh_mod_p}
		Let $\cM \in \QCoh_{W_nX}^{C^r}$. Then the following statements are equivalent:
		\begin{enumerate}
			\item $\cM$ is ind--coherent;
			\item $\cM/p$ is ind--coherent;
			\item $\cM[p]$ is ind--coherent.
		\end{enumerate}
	\end{lem}
	\begin{proof}
		Throughout, we will use \autoref{ind_coh_stable_under_everything} without further mention. If $\cM$ is ind--coherent, then surely both $\cM/p$ and $\cM[p]$ are. Now, assume that $\cM[p]$ is ind--coherent, and let us show that $\cM/p$ is also ind--coherent. For any $i \geq 0$, we have an exact sequence \[ 0 \to \quot{\cM[p] \cap p^{i}\cM}{\cM[p] \cap p^{i + 1}\cM} \to \quot{p^{i}\cM}{p^{i+1}\cM} \xto{p} \quot{p^{i+1}\cM}{p^{i+2}\cM}\to 0. \]
		The left term is always ind--coherent by assumption, and the right term is ind--coherent for $i \geq n - 1$ (it is zero). Thus, we deduce by descending induction that the middle term is ind--coherent for all $i \geq 0$. In particular, $\cM/p$ is ind--coherent.
		
		Finally, assume that $\cM/p$ is ind--coherent, and let us show that $\cM$ ind--coherent. Given $i \geq 0$, we have an exact sequence \[ 0 \to p^{i + 1}\cM \to p^i\cM \to \quot{p^i\cM}{p^{i + 1}\cM} \to 0. \]
		The right term is a quotient of $\cM/p$, so it is ind--coherent by assumption. Since the left term is also ind--coherent for $i \geq n$ (it is zero), we deduce by descending induction that the middle term is also ind--coherent. In particular, $\cM$ is ind--coherent.
	\end{proof}

%

	\begin{notation}
		From now on, we will denote by $X'$ be Noetherian $F$--finite $\bF_p$--scheme $(X, W_n\cO_X/p)$.
	\end{notation}
	
	\begin{prop}\label{inj_ind_coh_is_inj_qcoh}
		An injective object in $\IndCoh_{W_nX}^{C^r}$ is also injective in $\Mod(W_n\cO_X)$.
	\end{prop}
	\begin{proof}
		By the same proof as in \cite[Proposition 3.2.8]{Baudin_Duality_between_perverse_sheaves_and_Cartier_crystals}, it is enough to show the following: for any $\cM \in \Coh_{W_nX}^{C^r}$, if $\cI$ is an injective hull of $\cM$ in $\QCoh_{W_nX}$ and $\theta \colon F^r_*\cI \to \cI$ is any morphism making $(\cM, \kappa_{\cM}) \inj (\cI, \theta)$ a morphism of $W_n$--Cartier modules, then $(\cI, \theta)$ is ind--coherent. \\
		
		Let $\cM$ and $(\cI, \theta)$ be as in the statement above. First of all, let us show that $\cI[p]$ is an injective hull of $\cM[p]$ in $\QCoh_{X'}$. 
		
		Since the functor taking $p$--torsion is a right adjoint to the (exact) forgetful functor $\QCoh_{X'} \to \QCoh_{W_nX}$, it preserves injectives. Hence it remains to show that $\cM[p] \inc \cI[p]$ is an essential extension. However, if $\cN \inc \cI[p]$ is a non-zero submodule, then since in particular $\cN \inc \cI$, we deduce that $\cN \cap \cM \neq 0$. Since $\cN$ is $p$--torsion, we deduce that $\cN \cap \cM[p] \neq 0$, so the inclusion $\cM \inc \cI$ is indeed an essential extension. Thus, we have proven that $\cI[p]$ is an injective hull of $\cM[p]$ in $\QCoh_{X'}$. \\
		
		By \cite[Proposition 3.2.8]{Baudin_Duality_between_perverse_sheaves_and_Cartier_crystals} applied to the $\bF_p$--scheme $X'$, we deduce that $(\cI[p], \theta)$ is ind--coherent. We then conclude by \autoref{indcoh_iff_indcoh_mod_p} that also $(\cI, \theta)$ is ind--coherent.
	\end{proof}
	
	\begin{cor}\label{injectives_everywhere}
		Let $\cI$ be an injective object in either $\Mod(W_n\cO_X[C^r])$, $\QCoh_{W_nX}^{C^r}$, $\QCoh_{W_nX}^{C^r, \unit}$, $\IndCoh_{W_nX}^{C^r}$ or $\IndCoh_{W_nX}^{C^r, \unit}$. Then $\cI$ is also injective in $\Mod(W_n\cO_X)$.
	\end{cor}
	\begin{proof}
		By the same unitalization argument as in \cite[Corollary 3.4.9]{Baudin_Duality_between_perverse_sheaves_and_Cartier_crystals}, it is enough to show the result for $\Mod(W_n\cO_X[C^r])$, $\QCoh_{W_nX}^{C^r}$ and  $\IndCoh_{W_nX}^{C^r}$. For $\Mod(W_n\cO_X[C^r])$ and $\QCoh_{W_nX}^{C^r}$, apply exactly the same proof as in \cite[Proposition 3.1.5]{Baudin_Duality_between_perverse_sheaves_and_Cartier_crystals}. For $\IndCoh_{W_nX}^{C^r}$, this is exactly \autoref{inj_ind_coh_is_inj_qcoh}.
	\end{proof}

	Let us now go to our second goal. As already stated, we will again reduce to the case when $n = 1$ (something that we do not know how to do for $W_n$--Frobenius modules).
	
	\begin{defn}
		Define the functor $\ind_{W_nX} \colon \QCoh_{W_nX}^{C^r} \ra \IndCoh_{W_nX}^{C^r}$ by sending a $W_n$--Cartier module to its maximal ind--coherent $W_n$--Cartier submodule. This defines a functor, which is right adjoint to the inclusion $G \colon \IndCoh_{W_nX}^{C^r} \ra \QCoh_{W_nX}^{C^r}$. 
	\end{defn}

	\begin{lem}\label{first_commutativity}
		The following square commutes:
		\[ \begin{tikzcd}
			D^+(\QCoh_{W_nX}^{C^r}) \arrow[rr, "R\ind_{W_nX}"] \arrow[d, "i_{X'}^!"] &  & D^+(\IndCoh_{W_nX}^{C^r}) \arrow[d, "i_{X'}^!"] \\
			D^+(\QCoh_{X'}^{C^r}) \arrow[rr, "R\ind_{X'}"]                           &  & D^+(\IndCoh_{X'}^{C^r}).                       
		\end{tikzcd} \]
	\end{lem}
	\begin{proof}
		Since $i_{X'}^{\flat}$ is a right adjoint of the inclusion $\QCoh_{X'}^{F^r} \inj \QCoh_{W_nX}^{F^r}$, we know by \stacksproj{0DVC} that $i_{X'}^!$ is a right adjoint of $D^+(\QCoh_{X'}^{F^r}) \to D^+(\QCoh_{W_nX}^{F^r})$ (and similarly for $\IndCoh$). We have a similar picture for the $\ind$ functors, so in order to check that the diagram commutes, it is enough to check it for their adjoints. Thus, we must show that the diagram 
		\[ \begin{tikzcd}
				D^+(\QCoh_{W_nX}^{C^r})                        &  & D^+(\IndCoh_{W_nX}^{C^r}) \arrow[ll, "G"']                        \\
		  		D^+(\QCoh_{X'}^{C^r}) \arrow[u, "{i_{X', *}}"] &  & D^+(\IndCoh_{X'}^{C^r}) \arrow[u, "{i_{X', *}}"'] \arrow[ll, "G"]
	  	\end{tikzcd} \] commutes. This is immediate.
	\end{proof}

	Although the previous lemma works exactly the same way for $W_n$--Frobenius modules, we do not know about the following one.

	\begin{lem}\label{second_commutativity}
		The diagram \[ \begin{tikzcd}
			D^+(\IndCoh_{W_nX}^{C^r}) \arrow[d, "i_{X'}^!"] \arrow[rr, "G"] &  & D^+(\QCoh_{W_nX}^{C^r}) \arrow[d, "i_{X'}^!"] \\
			D^+(\IndCoh_{X'}^{C^r}) \arrow[rr, "G"]                         &  & D^+(\QCoh_{X'}^{C^r}).                        
		\end{tikzcd} \] commutes.
	\end{lem}
	\begin{proof}
		This diagram certainly commutes on the non--derived level, so given that $G$ is exact, it is enough to show that $G$ sends injectives to acyclic objects for $i_{X'}^{\flat} \colon \QCoh_{W_nX}^{C^r} \to \QCoh_{X'}^{C^r}$. Let $\cI \in \IndCoh_{W_nX}^{C^r}$ be an injective object. Since $\QCoh_{W_nX}^{C^r} \to \QCoh_{W_nX}$ preserves injective objects (\autoref{injectives_everywhere}), it is enough to show that $\cI$ is acyclic for $i_{X'}^{\flat} \colon \QCoh_{W_nX} \to \QCoh_{X'}$. However, $\cI$ is injective in $\QCoh_{W_nX}$ by \autoref{injectives_everywhere}, so we are done.
	\end{proof}
	
	\begin{cor}\label{ind-coh is ind-acyclic}
		Let $\Mcal \in \IndCoh_{W_nX}^{C^r}$. Then for all $i > 0$, $R^i\ind_X(\Mcal) = 0$.
	\end{cor}
	\begin{proof}
		By adjunction (see \stacksproj{0DVC}), we obtain the natural counit morphism $\theta \colon G(R\ind_{W_nX}(\cM)) \to \cM$. We will show that it is an isomorphism. Note that by \autoref{first_commutativity} and \autoref{second_commutativity}, we have a natural isomorphism \[ i_{X'}^!G(R\ind_{W_nX}(\cM)) \cong G(R\ind_{X'}(i_{X'}^!(\cM))).\]
		Since $G(R\ind_{X'}(i_{X'}^!(\cM))) \cong i_{X'}^!(\cM)$ (apply a hypercohomology spectral sequence argument with \cite[Corollary 3.2.23]{Baudin_Duality_between_perverse_sheaves_and_Cartier_crystals}), we deduce the morphism $\theta \colon G(R\ind_{W_nX}(\cM)) \to \cM$ becomes an isomorphism after applying $i_{X'}^!$ . Assume by contradiction that $\theta$ is not an isomorphism, and let $\cC$ denote its cone. Then $i_{X'}^!\cC = 0$ by our work above, so if $j \geq 0$ is the smallest integer such that $\cH^j(\cC) \neq 0$, we deduce that \[ \cH^j(i_{X'}^!\cC) = \cH^j(\cC)[p] = 0. \] This is impossible, since any non--zero $W_n\cO_X$--module must have non--trivial $p$--torsion. 
	\end{proof}
	
	\begin{prop}\label{D^b(Coh) = D^b_{coh}(QCoh) for Cartier modules}
		For any $* \in \{+, b\}$, the functors \[ D^*(\IndCoh_{W_nX}^{C^r}) \ra D^*_{\indcoh}(\QCoh_{W_nX}^{C^r}) \] and \[ D^b(\Coh_{W_nX}^{C^r}) \ra D^b_{\coh}(\QCoh_{W_nX}^{C^r}) \] are equivalences of categories.
	\end{prop}
	\begin{proof}
		Thanks to \autoref{ind-coh is ind-acyclic}, the proof is identical to that of \cite[Corollary 3.2.24]{Baudin_Duality_between_perverse_sheaves_and_Cartier_crystals}.
	\end{proof}
	
	\subsection{Witt--Cartier crystals}
	\begin{defn}\label{def_crystals}
		Let $\Mcal$ be a quasi-coherent $W_n$--Cartier module on $X$.
		\begin{itemize}
			\item It is said to be \emph{nilpotent} if $\kappa_{\cM}^n = 0$ for some $n \geq 0$, where we do the abuse of notations \[ \kappa_{\cM}^n \coloneqq \kappa_{\cM} \circ F^r_*\kappa_{\cM} \circ \dots \circ F^{r(n - 1)}_*\kappa_{\cM}. \] A morphism of $W_n$--Cartier modules whose kernel and cokernel are nilpotent is called a \emph{nil-isomorphism}.
			\item It is said to be \emph{locally nilpotent} if it is a union of nilpotent $W_n$--Cartier modules. A morphism of $W_n$--Cartier modules whose kernel and cokernel are locally nilpotent is called an \emph{lnil-isomorphism}. The full subcategory of locally nilpotent $W_n$--Cartier modules is denoted $\LNil$.
		\end{itemize}
	\end{defn}

	\begin{lem}\label{loc_nilp_is_ind_coh}
		Let $\cM$ be a locally nilpotent $W_n$--Cartier module. Then $\cM$ is ind--coherent.
	\end{lem}
	\begin{proof}
		Since $\cM$ is locally nilpotent, so is the Cartier module $\cM[p]$ on $X'$. By the case when $n = 1$ (see e.g. \cite[Remark 3.2.6]{Baudin_Duality_between_perverse_sheaves_and_Cartier_crystals} or \cite[Corollary 2.1.4]{Blickle_Bockle_Cartier_crystals}), we deduce that $\cM[p]$ is ind--coherent. We conclude the proof by \autoref{indcoh_iff_indcoh_mod_p}.
	\end{proof}

	\begin{defn}
		We set $\LNilCoh \inc \QCoh_{W_nX}^{C^r}$ to be the full subcategory of objects $\cM$ such that $\cM \sim_C \cN$ for some $\cN \in \Coh_{W_nX}^{C^r}$.
	\end{defn}

	\begin{lem}\label{LNilCoh closed under exts}
		We have that $\LNilCoh \inc \IndCoh_{W_nX}^{C^r}$, and that $\LNilCoh$ is closed under subobjects, quotients and extensions.
	\end{lem}
	\begin{proof}
		The first assertion follows immediately from \autoref{loc_nilp_is_ind_coh} and \autoref{ind_coh_stable_under_everything}. Let us prove the second one. In fact, is it straightforward once we show that $\cM \in \LNilCoh$ if and only if there exists a coherent $W_n$--Cartier submodule $\cM' \inc \cM$ such that $\cM/\cM'$ is in $\LNil$. The ``if'' direction is immediate. To show that the ``only if'' direction holds too, fix $\cM \in \LNilCoh$. By hypothesis, there exists an isomorphism $f \colon \cN \dra \cM$ in $\IndCrys_{W_nX}^{C^r}$. By the same argument as in \cite[Proposition 3.4.6]{Bockle_Pink_Cohomological_Theory_of_crystals_over_function_fields}, there exists a commutative diagram 
		\[ \begin{tikzcd}
			F^{er}_*\cN \arrow[d, "\kappa_{\cN}^e"'] \arrow[rrd, "g"] &  &     \\
			\cN \arrow[rr, "f"', dashed]                              &  & \cM,
		\end{tikzcd} \] 
		where the morphism $g$ is a honest morphism of $W_n$--Cartier modules. Since $g$ is automatically an lnil--isomorphism, taking $\cM' = \im(g) \inc \cM$ gives that $\cM/\cM' \sim_C 0$.
	\end{proof}

	\begin{defn}
		\begin{itemize}
			\item We define $\QCrys_{W_nX}^{C^r} \coloneqq \QCoh_{W_nX}^{C^r}/\LNil$ and $\IndCrys_{W_nX}^{C^r} \coloneqq \IndCoh_{W_nX}^{C^r}/\LNil$ (see \autoref{loc_nilp_is_ind_coh}). We also define the category of $W_n$--Cartier crystals to be $\Crys_{W_nX}^{C^r} \coloneqq \Coh_{W_nX}^{C^r}/\Nil$.
			\item A morphism in either $\QCrys_{W_nX}^{C^r}, \IndCrys_{W_nX}^{C^r}$ or $\Crys_{W_nX}^{C^r}$ will be denoted by dashed arrows $\cM \dra \cN$. 
			\item If $\cM$, $\cN \in \QCoh_{W_nX}^{C^r}$ are objects which become isomorphic in $\QCrys_{W_nX}^{C^r}$, we write $\cM \sim_C \cN$. We use the same notation for complexes in the derived category.
		\end{itemize}
	\end{defn}
	
	\begin{lem}\label{qcrys and incrys are Groth cats}
		Both $\IndCrys_{W_nX}^{C^r}$ and $\QCrys_{W_nX}^{C^r}$ are Grothendieck categories. In particular, they have enough injectives.
	\end{lem}
	\begin{proof}
		The proof is identical to that of \cite[Lemma 3.3.5]{Baudin_Duality_between_perverse_sheaves_and_Cartier_crystals}.
	\end{proof}
	
	\begin{cor}\label{qcrys with crys cohom is crys}
		For $* \in \{+, b\}$ the natural functors \[ D^*(\IndCrys_{W_nX}^{C^r}) \ra D^*_{\indcrys}(\QCrys_{W_nX}^{C^r}) \] and \[ D^b(\Crys_{W_nX}^{C^r}) \ra D^b_{\crys}(\QCrys_{W_nX}^{C^r}) \] are equivalences, where the subscript crys in $D^*_{\crys}(\QCrys_{W_nX}^{C^r})$ means that cohomology sheaves of the complex must lie in the essential image of $\Crys_{W_nX}^{C^r} \ra \QCrys_{W_nX}^{C^r}$.
	\end{cor}
	
	\begin{proof}
		Thanks to \autoref{D^b(Coh) = D^b_{coh}(QCoh) for Cartier modules}, the proof is identical to that of \cite[Corollary 3.3.8]{Baudin_Duality_between_perverse_sheaves_and_Cartier_crystals}.
	\end{proof}

	\begin{lem}\label{Crys_supported_at_p^m_is_W_m_Cartier}
		Let $1 \leq m \leq n$ be an integer. Then the canonical functors $\Crys_{W_mX}^{C^r} \to \Crys_{W_nX}^{C^r}$ and $\IndCrys_{W_mX}^{C^r} \to \IndCrys_{W_nX}^{C^r}$ (see \autoref{Crys_supported_at_p^m_is_W_m}) are fully faithful, and their essential image of objects $\cM$ such that $p^m\cM \sim_C 0$.
	\end{lem}
	\begin{proof}
		The proof is the same as that of \autoref{Crys_supported_at_p^m_is_W_m}, where instead of the commutative diagram in \emph{loc. cit.}, we have the diagram \[ \begin{tikzcd}
			\cM \arrow[rr, "f", dashed]                              &  & \cN. \\
			F^{er}_*\cM \arrow[u, "\kappa^e_{\cM}"] \arrow[rru, "g"] &  &    
		\end{tikzcd} \]
	\end{proof}

	\subsection{Finiteness theorems and fixed points of Cartier operators}

	Now, let us move to the important finiteness properties of $W_n$--Cartier modules and crystals. They will all follow from the case of usual Cartier modules.
	
	\begin{defn}
		Let $\cM \in \Mod(W_n\cO_X[C^r])$. The \emph{perfection} of $\cM$ is by definition \[ \cM^{\perf} \coloneqq \lim F^{er}_*\cM. \]
	\end{defn}
	
	\begin{lem}\label{Gabber_finiteness_prop}
		Let $\cM$ be a coherent $W_n$--Cartier module on $X$.
		\begin{enumerate}
			\item The inverse system $\{F^{er}_*\cM\}_{e \geq 1}$ satisfies the Mittag--Leffler condition. In particular, $\cM^{\perf} = \Rlim F^{er}_*\cM$ and perfection is exact on coherent $W_n$--Cartier modules.
			\item For any proper morphism $\pi \colon X \to Y$ of Noetherian $F$--finite $\bF_p$--schemes, it holds that $(R^i\pi_*\cM)^{\perf} \cong R^i\pi_*\cM^{\perf}$.
			\item We have that $\cM \sim_C 0$ if and only if $\cM^{\perf} = 0$.
			\item For any morphism $f \colon \cM \to \cN$ of coherent $W_n$--Cartier modules, $f$ is a nil--isomorphism if and only if $f^{\perf}$ is an isomorphism.
		\end{enumerate}
	\end{lem}
	\begin{proof}
		\begin{enumerate}
			\item Consider the exact sequence of coherent $W_n$--Cartier modules \[ 0 \to p\cM \to \cM \to \cM/p\cM \to 0. \]
			Since $\cM/p\cM$ is a Cartier module over $X'$, we know by \cite[Lemma 13.1]{Gabber_notes_on_some_t_structures} that the Mittag--Leffler condition holds for the system $\{F^{er}_*(\cM/p\cM)\}_e$, so it is enough to show that it also holds for the system $\{F^{er}_*(p\cM)\}_e$. Repeating the same procedure and using that the result trivially holds for $p^n\cM = 0$, we conclude the proof.
			\item We have that \[ R\pi_*\cM^{\perf} \expl{\cong}{see the previous point} R\pi_* \Rlim_e F^{er}_*\cM \cong \Rlim_e F^{er}_*R\pi_*\cM. \] The statement now follows by taking $\cH^i$ and a spectral sequence argument, since all the systems $\{ F^{er}R^j\pi_*\cM \}_{e \geq 0}$ satisfy the Mittag--Leffler condition (hence are $\lim$--acyclic).
			\item This is immediate from the Mittag--Leffler property of the system $\{F^{er}_*\cM\}_e$.
			\item This follows from the previous point and exactness of perfection. $\qedhere$
		\end{enumerate}
	\end{proof}

	\begin{rem}
		One might think from \autoref{Gabber_finiteness_prop} that perfection of $W_n$--Cartier modules plays an analogous role as perfection for $W_n$--Frobenius modules. There are several issues though, relying on the fact that inverse limits are not very pleasant to work with. For example, the Cartier module $H^0(\bA^1_k, \omega_{\bA^1_k})$ over a perfect field $k$ is locally nilpotent, but its structural morphism is surjective, so perfection will not vanish!

		The correct analogue is the \emph{unitalization functor}, sending $\cM \in \QCoh_{W_nX}^{C^r}$ to $\cM^u \coloneqq \colim F^{er, \flat}\cM \in \QCoh_{W_nX}^{C^r, \unit}$. For example, it induces an equivalence between $\QCrys_{W_nX}^{C^r}$ and $\QCoh_{W_nX}^{C^r, \unit}$ (compare with \autoref{perfection_induces_iso_crys_and_hol}).
	\end{rem}

	Let us move to finiteness properties of pushforwards of ind--coherent $W_n$--Cartier modules (which completely fails for $W_n$--Frobenius modules). To deal with open immersions, we will give a different proof than that of \cite{Baudin_Duality_between_perverse_sheaves_and_Cartier_crystals}.

	\begin{prop}\label{finiteness_pushforwards}
		Let $f \colon X \to Y$ be a separated morphism of Noetherian $F$--finite $\bF_p$--schemes, and let $i \geq 0$. Then for any $\cM \in \IndCoh_{W_nX}^{C^r}$, we have that $R^if_*\cM \in \IndCoh_{W_nY}^{C^r}$. If furthermore $\cM \in \Coh_{W_nY}^{C^r}$, then there exists $\cN \in \Coh_{W_nY}^{C^r}$ such that $R^if_*\cM \sim_C \cN$.
	\end{prop}
	\begin{rem}
		By \autoref{injectives_everywhere}, there is no ambiguity in what we mean by derived pushforwards.
	\end{rem}
	
	We first deal with open immersions. To do so, let us introduce a new functor, which we will see again when studying the duality:
	
	\begin{constr}\label{case_pairing}
		Let $\cI \subseteq \cO_X$ be a quasi--coherent ideal, and let $\cM \in \Mod(W_n\cO_X[C^r])$. Let us define a $W_n$--Cartier module structure on $\HHom_{W_n\cO_X}(W_n\cI, \cM)$ by sending $f \colon W_n\cI \to \cM$ to the function sending $s \in W_n\cI$ to $\kappa_{\cM}(f(F^r(s))) \in \cM$.
	\end{constr}

	Since we have an endofunctor $\HHom(W_n\cI, -)$ on $W_n$--Cartier modules, we can take its derived functor. By \autoref{injectives_everywhere}, the underlying quasi--coherent sheaf agrees with $\EExt^i_{W_n\cO_X}(W_n\cI, -)$.

	\begin{lem}\label{finiteness_pushforward_open_immersion}
		Let $j \colon U \inj X$ be an open immersion of Noetherian $F$--finite $\bF_p$--schemes.
		\begin{enumerate}
			\item\label{cool_finiteness} Let $\cI$ denote the ideal of the closed subscheme $X \setminus U$ (with any scheme structure). Then for any $\cM \in \QCoh_{W_nX}^{C^r}$, there is a natural lnil--isomorphism \[ R^ij_*(\cM|_U) \sim_C \EExt^i(W_n\cI, \cM). \]
			\item For any $\cN \in \Coh_{W_nU}^{C^r}$, the $W_n$--Cartier module $R^ij_*\cN$ is lnil--isomorphic to a coherent $W_n$--Cartier module on $X$.
		\end{enumerate} 
	\end{lem}
	\begin{proof}
		\begin{enumerate}
			\item By taking derived functors, it is enough to deal with the case $i = 0$. For any $m \geq 1$, we have by adjunction between $j_*$ and $(\cdot)|_U$ a natural map $\HHom_{W_n\cO_X}((W_n\cI)^m, \cM) \to  j_*\cM|_U$ of quasi--coherent sheaves. Explicitly, it sends $f \colon (W_n\cI)^m \to \cM$ to $f(1_U)$, where $1_U = (1, \dots, 1) \in W_n\cO_X(U) = \Gamma(U, (W_n\cI)^m)$ (whenever $1_U$ is defined). An immediate computation shows that it is compatible with the Cartier module structures, so we obtain a morphism of $W_n$--Cartier modules \[ \colim_m \HHom_{W_n\cO_X}((W_n\cI)^m, \cM) \to j_*\cM|_U. \] It is an isomorphism by \stacksproj{01YB}. To conclude, it is then enough to show that the natural map \[ \HHom_{W_n\cO_X}(W_n\cI, \cM) \to \colim_m \HHom_{W_n\cO_X}((W_n\cI)^m, \cM) \] is an lnil--isomorphism. Hence, it is enough to show that for a fixed $m \geq 1$, the map $\HHom_{W_n\cO_X}(W_n\cI, \cM) \to \HHom_{W_n\cO_X}((W_n\cI)^m, \cM)$ is an lnil--isomorphism. This is an immediate consequence of \autoref{Witt_ideal_inclusions}.
			\item By the previous point, it is enough to show that there exists a coherent $W_n$--Cartier module $\cM$ on $X$ such that $\cM|_U = \cN$. Since $j_*\cN[p]$ is ind--coherent (apply \cite[Corollary 3.2.16]{Baudin_Duality_between_perverse_sheaves_and_Cartier_crystals} or \cite[Corollary 3.2.5]{Blickle_Bockle_Cartier_crystals}) on $X'$), we deduce that \autoref{indcoh_iff_indcoh_mod_p} that $j_*\cN$ is ind--coherent. Thus, we obtain the existence of such an $\cM$ by Noetherianity. \qedhere
		\end{enumerate}
	\end{proof}

	\begin{proof}[Proof of \autoref{finiteness_pushforwards}]
		Since (derived) pushforwards preserve colimits, we may assume that $\cM$ is coherent. Note that it follows from \autoref{loc_nilp_is_ind_coh} and \autoref{ind_coh_stable_under_everything} that if two $W_n$--Cartier modules are lnil--isomorphic, then one is ind--coherent if and only if the other is. Thus, it is enough to show that each $R^if_*\cM$ is lnil--isomorphic to a coherent $W_n$--Cartier module. If $f$ is an open immersion, then the result follows from \autoref{finiteness_pushforward_open_immersion} and if $f$ is proper, this is immediate. In general, we know by \stacksproj{0F41} that $f$ can be written as a composition of two such maps, so we conclude the proof by a Leray spectral sequence argument and \autoref{LNilCoh closed under exts}.
	\end{proof}	

	\begin{prop}\label{Noeth_and_Art}
		Any object in $\Crys_{W_nX}^{C^r}$ is both Noetherian and Artinian.
	\end{prop}
	\begin{proof}
		For $1 \leq m \leq n$, let $\cC_m$ denote the full subcategory of $p^m$-torsion $W_n$--Cartier crystals. We show by ascending induction on $m$ that any element in $\cC_m$ is both Artinian and Noetherian.
	
		For $m = 1$, this follows from \cite[Main 	Theorem]{Blickle_Bockle_Cartier_modules_finiteness_results} (or \cite[Corollary 5.3.1]{Baudin_Duality_between_perverse_sheaves_and_Cartier_crystals}) and \autoref{Crys_supported_at_p^m_is_W_m_Cartier}. In general, for any $\cN \in \cC_{m + 1}$, consider the exact sequence \[ 0 \to p\cN \to \cN \to \cN/p\cN \to 0. \] Then $p\cN \in \cC_n$ and $\cN/p\cN \in \cC_1$, so the result follows by induction.
	\end{proof}

\noindent We conclude this section with miscellaneous results about fixed points of Cartier operators that we will need in \cite{Baudin_Witt_GR_vanishing_and_applications}. 

\begin{situation}
	Until the end of this section, $X$ denotes a separated scheme of finite type over a perfect field $k$, with structural morphism $\pi \colon X \to \Spec k$. We will write $H^0(X, -) \coloneqq \pi_*(-)$.
\end{situation}

\begin{lem}\label{local_existence_of_ss_sections}
	Let $x \in X$, and let $\cM \in \IndCoh_{W_nX}^{C^r}$ such that $\cM_x \not\sim_C 0$. Then there exists an open neighborhood $U$ of $x$ such that $H^0(U, \cM) \not\sim_C 0$. 
\end{lem}
\begin{proof}
	This follows immediately from \autoref{weird_qproj_lemma}.
\end{proof}

\begin{rem}
	If we did not work with Cartier crystals but with mere quasi--coherent sheaves, this result would be trivial: take any affine open containing $x$. Here, the situation is much more subtle: if $X = \bA^1$, then a direct explicit computation shows that $H^0(\bA^1, \omega_{\bA^1}) \sim_C 0$, while surely $\omega_{\bA^1} \not\sim_C 0$.
	
	Nevertheless, $H^0(\bA^1 \setminus \{0\}, \omega_{\bA^1}) \not\sim_C 0$, since it contains the logarithmic form $\frac{dx}{x}$ fixed by the Cartier operator.
\end{rem}

\begin{lem}\label{weird_qproj_lemma}
	Assume that $X$ is affine, let $x \in X$ and let $\cM \in \IndCoh_{W_nX}^{C^r}$ such that $\cM \not\sim_C 0$. Then there exists an affine open neighborhood $U$ of $x$ such that $H^0(U, \cM) \not\sim_C 0$. 
\end{lem}
\begin{proof}
	Since $H^0$ preserves injections, we may assume that $\cM$ is coherent. Furthermore, it is enough to show the result for $\cM[p] \inc \cM$, so up to replacing $X$ by $X'$, we may assume that $n = 1$.
	
	We will show this result by induction on $d \coloneqq \dim \Supp_{\crys}(\cM)$ (see \cite[Definition 3.3.11]{Baudin_Duality_between_perverse_sheaves_and_Cartier_crystals}). If $d = 0$, then let us show that $U = X$ works. As in the end of the proof of \cite[Lemma 5.2.4]{Baudin_Duality_between_perverse_sheaves_and_Cartier_crystals}, there exists a finite set of points $Z \subseteq X$ and $\cM' \in \Coh_Z^{C}$ such that $i_{Z, *}\cM' \sim_C \cM$. In particular, $\cM' \not\sim_C 0$, so $H^0(Z, \cM') \not\sim_C 0$ since $Z$ is finite (it is immediate to see that if $h \colon Y \to W$ is a finite morphism of Noetherian $F$--finite schemes and $\cP \in \Coh_Y^C$, then $\cP \sim_C 0$ if and only if $h_*\cP \sim_C 0$). Thus, we obtain that $H^0(X, \cM) \sim_C H^0(Z, \cM') \not\sim_C 0$. 
	
	Assume now that $d > 0$. Set $X = \Spec R$, and let $\cM$ correpond to the $R$--module $M$. For $f \in R$, set $V(f) = \Spec(R/f) \subseteq X$ and $D(f) = X \setminus V(f)$.  \\
	
	We claim that for general $f \in R$ (i.e. $V(f)$ is not contained in some prescribed proper closed subset of $\Spec(R)$), the localization map $M \to M_f$ is not an lnil--isomorphism.
	
	Let us prove this claim. For any $f \in R$ such that $M \to M_f$ is an lnil--isomorphism, then \autoref{finiteness_pushforward_open_immersion}.\autoref{cool_finiteness} and \cite[Remark 5.3.6]{Baudin_Duality_between_perverse_sheaves_and_Cartier_crystals} show that $i_{V(f)}^!M \sim_C 0$. In particular, for all closed point $y \in V(f)$, we have that $i_y^!M \sim_C 0$. However, we know that for $i_y^!M \not\sim_C 0$ for a general closed point $y \in \Supp_{\crys}(\cM)$ (this is a consequence of \cite[Lemmas 2.2.17 and 5.1.10]{Baudin_Duality_between_perverse_sheaves_and_Cartier_crystals}). Hence, we obtain a contradiction since $f$ is general, so the claim is proven. \\
	
	Let $f \in R$ be an element such that the natural map $M \to M_f$ is injective, that $x \notin V(f)$, that the map $M \to M_f$ is a lnil--isomorphism at all generic points of $\Supp_{\crys}(\cM)$ of maximal dimension, and that the assertion of the above claim holds. We then have a short exact sequence \[ 0 \to M \to M_f \to N \to 0 \] with $N \not\sim_C 0$. Since $M_f$ is lnil-isomorphic to a coherent Cartier module by \autoref{finiteness_pushforward_open_immersion}.\autoref{cool_finiteness}, we obtain the existence of some $N' \in \Coh_X^{C^r}$ such that $N' \sim_C N$. Note that by the choice of $f$, we obtain that $\dim \Supp_{\crys}(N') < d$ and that $N' \not\sim_C 0$.
	
	By the induction hypothesis, there exists an affine open $V \inc X$ containing $x$ such that $H^0(V, N') \not\sim_C 0$. Since $M_f|_V$ surjects (as crystals) onto $N'|_V$ and $V$ is affine, we deduce that also $H^0(D(f) \cap V, M) = H^0(V, M_f) \surj H^0(V, N') \not\sim_C 0$, so the proof is complete.
\end{proof}
\begin{rem}
	The same proof works verbatim when $k$ is only assumed to be $F$--finite.
\end{rem}

Let us now define a solution functor on $W_n$--Cartier modules, analogous to that on $W_n$--Frobenius modules. First of all, let us see how to extend a $W_n$--Cartier module to the étale site.

\begin{constr}\label{extending_Cartier_modules_etale_site}
	Let $f \colon U \to X$ be an étale morphism of Noetherian and $F$--finite schemes. Then by \cite[Corollary A.18]{Langer_Zink_De_Rham_Witt_cohomology_for_a_proper_and_smooth_morphism}, the square \[ \begin{tikzcd}
		W_nX \arrow[r, "f"] \arrow[d, "F"'] & W_nY \arrow[d, "F"] \\
		W_nX \arrow[r, "f"'] & W_nY
	\end{tikzcd} \] is Cartesian. By flat base change, we deduce that for any $\cM \in \QCoh_{W_nX}$, the natural morphism $f^*F_*\cM \to F_*f^*\cM$ is an isomorphism. Hence, we naturally obtain a pullback functor $f^* \colon \QCoh_{W_nX}^{C^r} \to \QCoh_{W_nU}^{C^r}$. In other words, any quasi--coherent $W_n$--Cartier module can be extended to the étale site. We denote this construction as $\cM \mapsto \cM^{\et}$. Note that again by \cite[Corollary A.18]{Langer_Zink_De_Rham_Witt_cohomology_for_a_proper_and_smooth_morphism}, this construction commutes with the inclusions $\QCoh_{W_mX}^{C^r} \inj \QCoh_{W_nX}^{C^r}$ for $1 \leq m \leq n$.
\end{constr}

\begin{defn}
	Let $(\cM, \kappa_{\cM}) \in \QCoh_{W_nX}^{C^r}$. We define $\Sol(\cM) \coloneqq \ker(\kappa_{\cM}^{\et} - 1 \colon\cM^{\et} \to \cM^{\et}) \in \Sh(X_{\et}, W_n(\bF_q))$.
\end{defn}

Let us establish some properties of this functor.

\begin{prop}\label{C - 1 surjective}
	Let $(\cM, \kappa_{\cM})$ be an ind--coherent $W_n$--Cartier module on $X$. Then the morphism $\kappa_{\cM}^{\et} - 1 \colon \cM^{\et} \to \cM^{\et}$ is surjective on the étale site of $W_n(\bF_q)$--sheaves.
\end{prop}
\begin{rem}
	This surely holds over $F$--finite fields in general, but we do not need such generality for our application (also our restriction makes the proof easier).
\end{rem}
\begin{proof}
	Let $\cM \in \IndCoh_{W_nX}^{C^r}$. Since we may assume that $X$ is affine, it enough to show the result for $\pi_*\cM$, where $\pi \colon X \to \Spec k$ denotes the structural map of $X$. In other words, by \autoref{finiteness_pushforwards}, we may assume that $X = \Spec k$. Since the question is étale local, we may assume that $k$ is algebraically closed.
	
	Using the exact sequences \[ 0 \to \cM[p] \to \cM \to \cM/\cM[p] \to 0, \] and induction, it is enough to show the result when $\cM$ is $p$--torsion. In other words, we may assume that $n = 1$. We may also asusme that $\cM$ is coherent. Let $\cN \inc \cM$ be the biggest nulpotent sub--Cartier module of $\cM$, and consider the exact sequence \[ 0 \to \cN \to \cM \to \cV \to 0 \] where $\cV \coloneqq \cM/\cN$. It is then enough to show the result for $\cN$ and $\cV$.
	
	\begin{itemize}
		\item Since $\cN$ is nilpotent (say $\kappa_{\cN}^e = 0$), then there is an explicit inverse of $\kappa_{\cN} - 1$, namely $-(1 + \kappa_{\cN} + \dots + \kappa_{\cN}^{e - 1})$. Hence, $\kappa_{\cN} - 1$ is even bijective. 
	
		\item Note that $\cV$ has no nilpotent sub--Cartier module, so the morphism $\cV \to F^{\flat}\cV$ is automatically injective. Since we work over a field, it must be an isomorphism, so we deduce that also $\kappa_{\cV} \colon F^r_*\cV \to \cV$ is also bijective ($k$ is perfect). Since $\kappa_{\cV} - 1 = \kappa_{\cV} \circ (1 - \kappa_{\cV}^{-1})$, it is enough to show that $1 - \kappa_{\cV}^{-1}$ is bijective. Now, $\kappa_{\cV}^{-1} \colon \cV \to F^r_*\cV$ defines a perfect coherent Frobenius module over an algebraically closed field, so we know by \cite[Corollary p.143]{Mumford_Abelian_Varieties} that it is isomorphic to the Frobenius module $(k, (\cdot)^q)^{\oplus n}$ for some $n \geq 0$. Thus, we can verify the result by hand using that $k$ is separably closed. \qedhere
	\end{itemize}
\end{proof}

\begin{cor}\label{properties of Sol}
	The functor $\: \Sol \colon \IndCoh_{W_nX}^{C^r} \to \Sh_{\et}(X, W_n(\bF_q))$ is exact and factors through $\IndCrys_{W_nX}^{C^r}$. Furthermore, for any $\cM \in \IndCrys_{W_nX}^{C^r}$, integer $i \geq 0$ and separated morphism $f \colon X \to Y$ of finite type schemes over $k$, we have an isomorphism \[ R^if_*(\Sol(\cM)) \cong \Sol(R^if_*(\cM)).  \] 
\end{cor}
\begin{proof}
	We first start by showing that $\Sol$ is exact, so consider an exact sequence \[ \begin{tikzcd}
		0 \arrow[r] & \cM_1 \arrow[r] & \cM_2 \arrow[r] & \cM_3 \arrow[r] & 0
	\end{tikzcd} \] of ind--coherent $W_n$--Cartier modules. Then we have a comutative diagram of exact sequences \[ \begin{tikzcd}
		0 \arrow[r] & \cM_1^{\et} \arrow[r] \arrow[d, "\kappa - 1"', no head] & \cM_2^{\et} \arrow[r] \arrow[d, "\kappa - 1"'] & \cM_3^{\et} \arrow[r] \arrow[d, "\kappa - 1"'] & 0 \\
		0 \arrow[r] & \cM_1^{\et} \arrow[r]                                   & \cM_2^{\et} \arrow[r]                          & \cM_3^{\et} \arrow[r]                         & 0,
	\end{tikzcd}  \] so we deduce that $\Sol$ is exact by \autoref{C - 1 surjective} and the snake lemma.
	
	We already showed in the proof of \autoref{C - 1 surjective} that $\kappa - 1$ is bijective on a nilpotent Cartier module, and the same proof shows that this still holds for a $W_n$--Cartier module. Hence, this also holds for a locally nilpotent Cartier module, so we obtain that $\Sol(\cM) = 0$ for any $\cM \in \LNil$. We obtain our sought factorization by \stacksproj{02MS}.
	
	Let us finish with the commutativity property with higher pushforwards. The fact that $\Sol$ and $f_*$ commute follow immediately from the definitions. By exactness of $\Sol$, it is enough to show that if $\cI$ in injective in $\IndCoh_{W_nX}^{C^r}$, then $\Sol(\cI)$ is $f_*$--acyclic (we are implicitly using \autoref{finiteness_pushforwards}).
	
	Since $\cI$ is injective in $\QCoh_{W_nX}$ by \autoref{inj_ind_coh_is_inj_qcoh}, we know that $\cI^{\et}$ is acyclic for $f_* \colon \Sh_{\et}(X) \to \Sh_{\et}(Y)$ by \cite[Proposition III.3.8]{Milne_Etale_Cohomology}. Given that we have a short exact sequence of étale sheaves \[ \begin{tikzcd}
		0 \arrow[r] & \Sol(\cI) \arrow[r] & \cI^{\et} \arrow[r, "\kappa - 1"] & \cI^{\et} \arrow[r] & 0,
	\end{tikzcd} \] and that $\kappa - 1 \colon f_*\cI^{\et} \to f_*\cI^{\et}$ is also surjective (use \autoref{finiteness_pushforwards} and \autoref{C - 1 surjective}), we deduce that $R^if_*\Sol(\cI) = 0$ for all $i > 0$.
\end{proof}

	\section{The duality}
	
 	Fix a Noetherian $F$--finite $\bF_p$--scheme $X$.
	
	\subsection{Pairings, coherent duality and compatibilites}
	
	\begin{notation}\label{notation_iota}
		Throughout, we denote by $\iota$ either of the forgetful functor $\IndCoh_{W_nX}^{C^r, \unit} \to \QCoh_{W_nX}$ or $\IndCoh_{W_nX}^{C^r, \unit} \to \Mod(W_n\cO_X)$ (recall that it is only left--exact in general).
	\end{notation}

	\begin{constr}\label{construction_pairings}
		\begin{itemize}
			\item Let $\Mcal \in \Mod(W_n\Ocal_X[F^r])$ and $\Ncal \in \Mod(W_n\Ocal_X[C^r])$. Define the following $W_n$--Cartier module structure on $\HHom_{W_n\cO_X}(\Mcal, \Ncal)$: \[ \left(f \colon \Mcal|_U \ra \Ncal|_U \right) \mapsto \left(\kappa_{\cN} \circ F^r_*f \circ \tau_{\cM} \colon \Mcal|_U \ra \Ncal|_U \right) \] for all opens $U \inc X$. Due to the presence of $F^r_*f$ in the formula, this morphism is indeed $q^{-1}$-linear. We will often drop the symbol $W_n\cO_X$ is $\HHom_{W_n\cO_X}(-, -)$.
			\item Let $\Mcal \in \QCoh_{W_nX}^{C^r}$ and $\Ncal \in \QCoh_{W_nX}^{C^r, \unit}$. Define the following $W_n$--Frobenius module structure on $\HHom(\Mcal, \Ncal)$: 
			\[ \left(f \colon \Mcal|_U \ra \Ncal|_U \right) \mapsto \left((\kappa_{\cN}^{\flat})^{-1} \circ F^{r, \flat}f \circ \kappa_{\cM}^{\flat} \colon \Mcal|_U \ra \Ncal|_U \right) \] for all opens $U \inc X$. Due to the presence of $F^{r, \flat}f$ in the formula, this morphism is indeed $q$-linear.
		\end{itemize}
	\end{constr}

	\begin{rem}
		The first pairing generalizes \autoref{case_pairing}, where the $W_n$--Frobenius module is $W_n\cI$ with $\cI$ a quasi--coherent ideal.
	\end{rem}
	
	\begin{prop}\label{pairings}
		\begin{enumerate}
			\item\label{restr_pairing} The two pairings defined above gives rise to functors \[ \HHom(-, -) \colon (\Coh_{W_nX}^{F^r})^{op} \times \IndCoh_{W_nX}^{C^r} \ra \QCoh_{W_nX}^{C^r} \] and \[ \HHom(-, -) \colon (\Coh_{W_nX}^{C^r})^{op} \times \IndCoh_{W_nX}^{C^r, \unit} \ra \QCoh_{W_nX}^{F^r}. \]
			\item\label{derived_pairing} Taking derived functors gives rise to functors \[ \cR\HHom(-, -) \colon D^-(\Coh_{W_nX}^{F^r})^{op} \times D^+(\QCoh_{W_nX}^{C^r}) \ra D(\QCoh_{W_nX}^{F^r}) \] and \[ \cR\HHom(-, -) \colon D^-(\Coh_{W_nX}^{C^r})^{op} \times D^+(\IndCoh_{W_nX}^{C^r, \unit}) \ra D(\IndCoh_{W_nX}^{C^r}) \] which are triangulated when fixing either variables.
			\item\label{commutativity_with_qcoh} The diagrams \[ \begin{tikzcd}
				D^-(\Coh_{W_nX}^{F^r})^{op} \times D^+(\IndCoh_{W_nX}^{C^r}) \arrow[rr, "\Rcal\HHom"] \arrow[d] &  & D(\QCoh_{W_nX}^{C^r})  \arrow[d] \\
				D(W_n\Ocal_X)^{op} \times D^+(W_n\Ocal_{X}) \arrow[rr, "\Rcal\HHom"']                                &  & D(W_n\Ocal_{X})                           
			\end{tikzcd} \] and \[ \begin{tikzcd}
			D^-(\Coh_{W_nX}^{C^r})^{op} \times D^+(\IndCoh_{W_nX}^{C^r, \unit}) \arrow[rr, "\Rcal\HHom"] \arrow[d] &  & D(\QCoh_{W_nX}^{F^r})  \arrow[d] \\
			D(W_n\Ocal_{X})^{op} \times D^+(W_n\Ocal_{X}) \arrow[rr, "\Rcal\HHom"']                                &  & D(W_n\Ocal_{X})                           
		\end{tikzcd} \] are commutative, where the arrow $\IndCoh_{W_nX}^{C^r, \unit} \to D(W_n\cO_{X})$ in the last diagram is $R\iota$.
		\end{enumerate}
	\end{prop}
	\begin{proof}
		The first statement is immediate. The second statement is a formal consequence of the first one, and the last statement is a consequence of \autoref{injectives_everywhere}.
	\end{proof}

	Let $1 \leq m \leq n$ be an integer, and let $i_m \colon W_mX \inj W_nX$ denote the associated closed immersion. Then we have the usual pushforward functor, but recall that we also have the functor $i_m^{\flat} \colon \QCoh_{W_nX}^{C^r} \to \QCoh_{W_mX}^{C^r}$, since $i_m$ commutes with Frobenii (see a quasi--coherent $W_n$--Cartier module as the pair $(\cN, \kappa_{\cN}^{\flat})$ and apply $i_m^{\flat}$ to $\kappa_{\cN}^{\flat}$). Then $i_m^{\flat}$ is right adjoint to $i_{m, *}$ and by construction, this preserves unit $W_n$--Cartier modules.

	\begin{lem}\label{compatibility_different_n}
		Let $1 \leq m \leq n$ be an integer. Then for any $\cM^{\bullet} \in D^-(\Coh_{W_mX}^{F^r})$ and $\cN^{\bullet} \in D^+(\IndCoh_{W_nX}^{C^r})$, there is a natural isomorphism \[ i_{m, *}\cR\HHom_{W_m\cO_X}(\cM^{\bullet}, i_m^!\cN^{\bullet}) \cong \cR\HHom_{W_n\cO_X}(i_{m, *}\cM^{\bullet}, \cN^{\bullet}) \] in $D(\IndCoh_{W_nX}^{C^r})$.
	\end{lem}
	\begin{proof}
		Since $i_m^{\flat}$ is right adjoint to an exact functor, it preserves injective objects. Hence, by definition of $\cR\HHom$, we have to show that for $\cM \in \Coh_{W_nX}^{F^r}$ and $\cN \in \IndCoh_{W_nX}^{C^r}$, we have a natural isomorphism \[ i_{m, *}\HHom_{W_m\cO_X}(\cM, i_m^{\flat}\cN) \cong \HHom_{W_n\cO_X}(i_{m, *}\cM, \cN). \] 
		Consider the composition \[ i_{m ,*}\HHom(\cM, i_m^{\flat}\cM) \to \HHom(i_{m, *}\cM, i_{m, *}i_m^{\flat}\cN) \to \HHom(i_{m, *}\cM, \cN), \] where the first map is the natural map, and the second map is given by applying $\HHom(\cM, -)$ to the inclusion of $W_n$--Cartier modules $i_{m, *}i_{m}^{\flat}\cN \inc \cN$. Since the first map respects Cartier actions (it is immediate from the definitions), so does the composition. The fact that this is an isomorphism follows from the adjunction between $i_{m, *}$ and $i_m^{\flat}$ on quasi--coherent sheaves.
	\end{proof}
	\begin{rem}
		A straight--forward (but tedious) computation shows that this also holds for the pairing in the other side. Fortunately for us, we will be able to get around that fact.
	\end{rem}
	
	\begin{defn}\label{def property has a unit dualizing complex}
		A complex $W_n\omega_X^\bullet \in D^+(\IndCoh_{W_nX}^{C^r, \unit})$ is said to be a \emph{$W_n$--unit dualizing complex} if $R\iota(W_n\omega_X^\bullet) \in D(W_n\cO_X)$ is a dualizing complex (see \stacksproj{0A85}).
	\end{defn}
	\begin{rem}
		\begin{itemize}
			\item We will see in \autoref{easier_version_unit_dc} that the datum of a unit dualizing complex is equivalent to that of a dualizing complex $\cN \in D_{\coh}(\QCoh_X)$, together with an isomorphism $\cN \to F^{r, !}\cN$ in the derived category.
		\end{itemize}
	\end{rem}
	
	From now on, fix a unit dualizing complex $W_n\omega_X^\bullet$ on $X$ (we assume there exists at least one). Given any $\iota$ (see \autoref{notation_iota}), we make the abuse of notation ``$W_n\omega_X^\bullet = R\iota(W_n\omega_X^\bullet)$'' (this does not cause any trouble in practice, and simplifies notations).
	
	\begin{lem}\label{first properties of RHom(- , omega)}
		Pairing with $W_n\omega_X^\bullet$ restrict and corestrict to triangulated functors \[ \Rcal\HHom(-, W_n\omega_X^\bullet) \colon D^b(\Coh_{W_nX}^{F^r})^{op} \ra D^b_{\coh}(\QCoh_{W_nX}^{C^r}) \] and \[ \Rcal\HHom(-, W_n\omega_X^\bullet) \colon D^b(\Coh_{W_nX}^{C^r})^{op} \ra D^b_{\coh}(\QCoh_{W_nX}^{F^r}). \]
	\end{lem}
	\begin{proof}
		We have to show that given a complex $\cM^{\bullet}$ in $D^b(\Coh_{W_nX}^{F^r})$ or in $D^b(\Coh_{W_nX}^{C^r})$, the complex $\cR\HHom(\cM^{\bullet}, W_n\omega_X^{\bullet})$ lives in bounded degrees and has coherent cohomology sheaves. This can be checked in $D(W_n\cO_X)$, so we conclude by \autoref{pairings}.\autoref{commutativity_with_qcoh} and \stacksproj{0A89}.
	\end{proof}
	\begin{defn}\label{definition duality}
		Define the duality functors $\Dd_n$ as the following compositions:
		\[ \begin{tikzcd}
			D^b(\Coh_{W_nX}^{F^r})^{op} \arrow[rrr, "{\Rcal\HHom(-, W_n\omega_X^\bullet)}"] & & & {D^b_{\coh}(\QCoh_{W_nX}^{C^r})} \arrow[rrr, "\cong"] & & & D^b(\Coh_{W_nX}^{C^r})
		\end{tikzcd} \] and \[ \begin{tikzcd}
			D^b(\Coh_{W_nX}^{C^r})^{op} \arrow[rrr, "{\Rcal\HHom(-, W_n\omega_X^\bullet)}"] & & & {D^b_{\coh}(\QCoh_{W_nX}^{F^r})} \arrow[rrr, "\cong"] & & & {D^b(\Coh_{W_nX}^{F^r})},
		\end{tikzcd} \] where the equivalence $D^b_{\coh}(\QCoh_{W_nX}^{C^r}) \ra D^b(\Coh_{W_nX}^{C^r})$ comes from \autoref{D^b(Coh) = D^b_{coh}(QCoh) for Cartier modules}, and the equivalence $D^b_{\coh}(\QCoh_{W_nX}^{F^r}) \ra D^b(\Coh_{W_nX}^{F^r})$ comes from \autoref{D^b(Coh) = D^b_{coh}(QCoh) for F-modules}.
	\end{defn}
	
	\begin{lem}\label{evaluation map is okay sheaf case}
		Let $\Mcal \in \Coh_{W_nX}^{F^r}$ (resp. $\Coh_{W_nX}^{C^r}$) and $\Ncal \in \QCoh_{W_nX}^{C^r, \unit}$. Then the evaluation morphism \[ \ev \colon \Mcal \ra \HHom(\HHom(\Mcal, \Ncal), \Ncal) \] is a morphism of $W_n$--Frobenius modules (resp. $W_n$--Cartier modules).
	\end{lem}
	\begin{proof}
		The proof is identical to that of \cite[Lemma 4.2.6]{Baudin_Duality_between_perverse_sheaves_and_Cartier_crystals}.
	\end{proof}
	
	\begin{thm}\label{main thm duality}
		Let $X$ be a Noetherian, $F$-finite and semi--separated $\bF_q$-scheme with has a $W_n$--unit dualizing complex $W_n\omega_X^\bullet$. Then the two functors $\bD_n$ from \autoref{definition duality} are essential inverses to each other, so we obtain an equivalence
		 \[ D^b(\Coh_{W_nX}^{C^r})^{op} \cong D^b(\Coh_{W_nX}^{F^r}). \]
		 Furthermore, they also induce an equivalence \[ D^b(\Crys_{W_nX}^{C^r})^{op} \cong D^b(\Crys_{W_nX}^{F^r}). \]
		 Finally, given an integer $1 \leq m \leq n$, set $W_m\omega_X^{\bullet} \coloneqq i_m^!W_n\omega_X^{\bullet}$, with associated duality functors $\bD_m$ (either on modules or crystals). Then \[ i_{m, *} \circ \bD_m \cong \bD_n \circ i_{m, *}.  \]
	\end{thm}
	\begin{proof}
		The proof of the first statement is identical to \cite[Theorem 4.2.7]{Baudin_Duality_between_perverse_sheaves_and_Cartier_crystals}. Indeed, the ingredients in \emph{loc. cit.} that are used in the proof are: 
		\begin{itemize}
			\item the fact that the evaluation map respect the Frobenius/Cartier structures (see \autoref{evaluation map is okay sheaf case});
			\item the fact for $A \in \{F, C\}$, the functor $D^b(\Coh_{W_nX}^{A^r}) \to D^b_{\coh}(\QCoh_{W_nX}^{A^r})$ is an equivalence (see \autoref{D^b(Coh) = D^b_{coh}(QCoh) for F-modules} and \autoref{D^b(Coh) = D^b_{coh}(QCoh) for Cartier modules});
			\item the fact for $A \in \{F, C\}$, the functor $D^b(\Coh_{W_nX}^{A^r}) \to D^b_{\coh}(\cO_X[A^r])$ is an equivalence (see  \autoref{D^b(Coh) = D^b_{coh}(O_X) for F-modules} and \autoref{D(QCoh) = D_qcoh(O_X[C])}).
		\end{itemize}
	
		The proof of the second statement is identical to that of \cite[Theorem 4.3.5]{Baudin_Duality_between_perverse_sheaves_and_Cartier_crystals}. Let $(A, B) \in \{(F, C), (C, F)\}$, and let us show that there is a natural isomorphism \[  i_{m, *} \circ \bD_m \cong \bD_n \circ i_{m, *} \] as functors $D^b(\Coh_{W_mX }^{A^r})^{op} \to D^b(\Coh_{W_nX}^{B^r})$. If $A = F$, then this follows from \autoref{compatibility_different_n}. For $A = C$, we deduce the statement by the following string of isomorphisms: \[ \bD_n \circ i_{m, *} \cong \bD_n \circ i_{m, *} \circ \bD_m \circ \bD_m \expl{\cong}{by the case $A = F$} \bD_n \circ \bD_n \circ i_{m, *} \circ \bD_m \cong i_{m, *} \circ \bD_m. \]
	\end{proof}

	\begin{rem}
		Thanks to the last compatibility in \autoref{main thm duality}, we will remove the subscript $n$ in the notation $\bD_n$.
	\end{rem}

	By this duality theorem, we can proceed as in the case $n = 1$ to obtain a better understanding of $W_n$--unit dualizing complexes.
	
	\begin{prop}\label{easier_version_unit_dc}
		Assume that $X$ is semi--separated, and let $W_n\omega_X^{\bullet} \in D^b_{\coh}(W_n\cO_X)$ be a dualizing complex with a fixed isomorphism $W_n\omega_X^{\bullet} \to F^{r, !}W_n\omega_X^{\bullet}$. Then there exists a unique $W_n$--unit dualizing complex $\cM^{\bullet}$ that admits an isomorphism \[ \theta \colon R\iota\cM^{\bullet} \to W_n\omega_X^{\bullet} \] in $D(W_n\cO_X)$, such that the diagram 
		\[ \begin{tikzcd}
			R\iota\cM^{\bullet} \arrow[d, "\theta"'] \arrow[rr] &  & {F^{r, !}R\iota\cM^{\bullet}} \arrow[d, "{F^{r, !}\theta}"] \\
			W_n\omega_X^{\bullet} \arrow[rr]                                                &  & {F^{r, !}W_n\omega_X^{\bullet}}                                                        
		\end{tikzcd} \] commutes, where the top arrow is induced by the Cartier structure on $\cM^{\bullet}$.
	\end{prop}
	\begin{proof}
		Thanks to \autoref{main thm duality}, the result follows from the same proof as in \cite[Propositions 5.1.1 and 5.1.3]{Baudin_Duality_between_perverse_sheaves_and_Cartier_crystals}.
	\end{proof}
	
	\begin{rem}
		By \autoref{easier_version_unit_dc}, we deduce that the datum of a $W_n$--unit dualizing complex is equivalent to that or an ordinary dualizing complex $W_n\omega_X^{\bullet}$ with a fixed isomorphism $W_n\omega_X^{\bullet} \to F^{r, !}W_n\omega_X^{\bullet}$.
	\end{rem}

	\begin{cor}\label{existence_of_W_n_unit_dc}
		Let $X$ be a semi--separated scheme of finite type over the spectrum of a Noetherian, $F$--finite $\bF_q$--algebra. Then $X$ admits a $W_n$--unit dualizing complex.
	\end{cor}
	\begin{proof}
		Let $f \colon X \to \Spec(A)$ be a morphism of finite type, with $A$ a Noetherian, $F$--finite $\bF_q$--algebra. By \cite[Theorem 9.1]{Quasi_F_splittings_III}, there exists a dualizing complex $W_n\omega_A^{\bullet}$ on $\Spec(W_nA)$ with an isomorphism $W_n\omega_A^{\bullet} \to F^!W_n\omega_A^{\bullet}$. Applying the functor $f_{new}^!$ from \stacksproj{0AU5}, and setting $W_n\omega_X^{\bullet} \coloneqq f^!_{new}W_n\omega_A^{\bullet}$, we obtain an isomorphism \[ W_n\omega_X^{\bullet} \cong F^!W_n\omega_X^{\bullet}. \] Note that $W_n\omega_X^{\bullet}$ is a dualizing complex by the discussion in \emph{loc. cit.} (right above \stacksproj{0AU7}). Since we can compose the isomorphism above $r$ times, we conclude the proof by \autoref{easier_version_unit_dc}.
	\end{proof}

	\begin{notation}
		Let $f \colon X \to Y$ of finite type between semi--separated, Noetherian and $F$--finite $\bF_q$--schemes. Given a $W_n$--unit dualizing complex $W_n\omega_Y^{\bullet}$, we can proceed exactly as in \autoref{existence_of_W_n_unit_dc} and obtain a $W_n$--unit dualizing complex $W_n\omega_X^{\bullet}$ on $X$. We will write $W_n\omega_X^{\bullet} \coloneqq f^!W_n\omega_Y^{\bullet}$.
	\end{notation}

	As in the case $n = 1$, it is possible to define the functor $f^!$ for finite type morphisms on $W_n$--Cartier crystals using our duality (not only finite ones). Since we will not need it in practice and since there should be a much more natural definition using $\infty$--categories and results of \cite{Mattis_Weiss_The_derived_infty_category_of_Cartier_modules}, we will not write it explicitly.
	
	On the other hand, an important property is the compatibility between our duality functor and proper pushforwards.
	
	\begin{prop}\label{comp duality and pushforwards}
		Let $f \colon X \to Y$ be a proper morphism of finite type between Noetherian, $F$--finite and semi--separated $\bF_q$--schemes, and let $(A, B) \in \{(F, C), (C, F)\}$. Assume that $Y$ has a $W_n$--unit dualizing complex $W_n\omega_Y^{\bullet}$ with induced duality functor $\bD_{Y}$, and let $W_n\omega_X^{\bullet} \coloneqq f^!W_n\omega_Y^{\bullet}$, with induced duality functor $\bD_{X}$. Then we have that \[ Rf_* \circ \bD_X \cong \bD_Y \circ Rf_* \] as functors $D^b(\Coh_{W_nX}^{A^r})^{op} \to D^b(\Coh_{W_nY}^{B^r})$, where the $Rf_*$ functors are taken in $\IndCoh$.
	\end{prop}
	\begin{proof}
		The proof is identical to that of \cite[Corollary 5.1.7]{Baudin_Duality_between_perverse_sheaves_and_Cartier_crystals}.
	\end{proof}
	
	\subsection{Perverse $W_n(\Ff_q)$--sheaves and duality}\label{section eq cat}
	
	Our final goal is to show that $\Sol \circ \bD$ induces an equivalence of categories between $W_n$--Cartier crystals perverse $W_n(\bF_q)$--sheaves, by reducing to the case $n = 1$. Let us first define and work with this latter category. Unlike in \cite{Baudin_Duality_between_perverse_sheaves_and_Cartier_crystals}, we will not explicit a perverse t--structure on $W_n$--Frobenius modules and crystals. We believe that the right way to do this is to use techniques in the spirit of \cite{Mattis_Weiss_The_derived_infty_category_of_Cartier_modules}.
	
	Throughout, we will implicitly use that a Noetherian, $F$--finite $\bF_p$--scheme has finite dimension (see \cite[Proposition 1.1]{Kunz_Noetherian_Rings_of_char_p}).

	\begin{defn}[\cite{Gabber_notes_on_some_t_structures}]
		Define the \emph{perverse t-structure} on $D(X_{\et}, W_n(\bF_q))$ associated to the middle perversity to be given by \[ ^pD^{\leq 0} \coloneqq \bigset{\Fcal^{\bullet} \in D(X_{\et}, W_n(\Ff_q))}{\cH^j(\Fcal^{\bullet})_{\overline{x}} = 0 \esp \forall j > -\dim(\overline{\{x\}}) \esp \forall x \in X} \] and \[ ^pD^{\geq 0} \coloneqq \bigset{\Fcal^{\bullet} \in D^+(X_{\et}, W_n(\Ff_q))}{\cH^j(i_{\overline{\{x\}}}^!\Fcal^{\bullet})_{\overline{x}} = 0 \esp \forall j < -\dim(\overline{\{x\}}) \esp \forall x \in X}. \]
	\end{defn}

	\begin{prop}\label{non-constr_yet}
		The datum $({}^pD^{\leq 0}, {}^pD^{\geq 0})$ defines a t-stucture on $D(X_{\et}, W_n(\bF_q))$, and perverse truncations preserve bounded complexes. Hence, this t-structure restrics to a t-structure on $D^b(X_{\et}, W_n(\bF_q))$. Furthermore, for all $1 \leq m \leq n$, the natural functor $D^b(X_{\et}, W_m(\bF_q)) \to D^b(X_{\et}, W_n(\bF_q))$ is perverse t--exact.
	\end{prop}
	\begin{proof}
		The first statement is a direct consequence of the results in \cite[Section 6]{Gabber_notes_on_some_t_structures}. Let us show that perverse truncations preserve bounded complexes. Note that by definition, a perverse sheaf always lives in $D^b(X_{\et}, W_n(\bF_q))$. In particular, taking perverse cohomology sheaves surely preserves $D^b(X_{\et}, W_n(\bF_q))$. Let us deduce that also perverse truncations do, so let $\cF^{\bullet} \in D^b(X_{\et}, W_n(\cF_q))$. We first show that for $s \gg 0$, it holds that ${}^p\tau_{\geq s}(\cF^{\bullet}) = 0$ and ${}^p\tau_{\leq  -s}(\cF^{\bullet}) = 0$. Using the usual ``truncation exact triangles'', this would follow from the inclusion $\cF^{\bullet} \in {}^pD^{\leq s} \cap {}^pD^{\geq -s}$ for $s \gg 0$. The fact that $\cF^{\bullet} \in {}^pD^{\geq -s}$ is immediate from the definition, and the fact that $\cF^{\bullet} \in {}^pD^{\leq s}$ follows from finite--dimensionality of $X$.
		
		Now, let us get to the proof that perverse truncations preserve bounded complexes. We will only show that each ${}^p\tau_{\leq i}\cF^{\bullet}$ is bounded, as the argument for the complexes ${}^p\tau_{\geq i}\cF^{\bullet}$ is identical.
		
		By contradiction, let $i$ be the smallest integer such that ${}^p\tau_{\leq i}\cF^{\bullet} \notin D^b(X_{\et}, W_n(\bF_q))$ (such an integer exists, since ${}^p\tau_{\leq -s}(\cF^{\bullet}) = 0$ for $s \gg 0$). Then we have an exact triangle \[ {}^p\tau_{\leq j - 1}\cF^{\bullet} \to  {}^p\tau_{\leq j}\cF^{\bullet} \to {}^p\cH^j(\cF^{\bullet})[-j] \xto{+1} \] Since both the left and right term are bounded complexes, so is the middle term, so we are done.
		
		We are left to show that $D^b(X_{\et}, W_m(\bF_q)) \to D^b(X_{\et}, W_n(\bF_q))$ is perverse t--exact. The fact that it preserves ${}^pD^{\leq 0}$ is immediate from the definitions. To see that it also preserves ${}^pD^{\geq 0}$, it is enough to show that taking $i_{\overline{\{x\}}}^!$ in either $D(X_{\et}, W_m(\bF_q))$ or in $D(X_{\et}, W_n(\bF_q))$ gives to same result for an object $\cM^{\bullet} \in D(X_{\et}, W_m(\bF_q))$. This is immediate, since flasque sheaves are acyclic for $i_{\overline{\{x\}}}^!$ (see \cite[Proposition 6.6]{SGA_4_tome_2}), and being flasque only depends on the underlying abelian sheaf (see \cite[Beginning of 4.4]{SGA_4_tome_2}).
	\end{proof}
	
	\begin{cor}\label{perverse_t_structure}
		The perverse truncations and perverse cohomology sheaves functors on $D^b(X_{\et}, W_n(\bF_q))$ preserve the full subcategory $D^b_c(X_{\et}, W_n(\bF_q))$.
	\end{cor}
	\begin{proof}
		By the same argument as in the proof of \autoref{non-constr_yet}, it is enough to show that taking perverse cohomology sheaves preserve constructibility. Let $\cF^{\bullet} \in D^b_c(X_{\et}, W_n(\bF_q))$. Using the usual triangles \[ \tau_{\leq j}\cF^{\bullet} \to \cF^{\bullet} \to \tau_{\geq j + 1}\cF^{\bullet} \xto{+1} \] and the associated long exact sequence in perverse cohomology sheaves, we may assume that $\cF^{\bullet}$ is concentrated in a single degree, say $\cF^{\bullet} = \cF[s]$ for some $\cF \in \Sh^c(X_{\et}, W_n(\bF_q))$. Take a filtration $0 = \cF_0 \inc \cF_1 \inc \dots \inc \cF_n = \cF$ such that each $\cF_{i + 1}/\cF_i$ is $p$--torsion (hence an element in $\Sh(X_{\et}, \bF_q)$. Since the result holds for $p$--torsion sheaves by perverse t--exactness of $D^b(X_{\et}, \bF_q) \to D^b(X_{\et}, W_n(\bF_q))$ (see \autoref{non-constr_yet}) and the case $n = 1$ (see \cite[Theorem 10.3]{Gabber_notes_on_some_t_structures}), we deduce by induction and the long exact sequences in perverse cohomology sheaves that the result holds for $\cF$.		
	\end{proof}

	\begin{notation}
		The heart of the perverse t--structure on $D^b(X_{\et}, W_n(\bF_q))$ (resp. on $D^b_c(X_{\et}, W_n(\bF_q))$) is denoted $\Perv(X_{\et}, W_n(\bF_q))$ (resp. $\Perv_c(X_{\et}, W_n(\bF_q))$).
	\end{notation}

	We now have a perverse t--structure on $D^b_c(X_{\et}, W_n(\bF_q))$, and our goal is to show that it corresponds to the standard t--structure on $D^b(\Crys_{W_nX}^{C^r})$ via $\Sol \circ \: \bD$. Again, the proof will go by reducing to the case $n = 1$.
	
	\begin{lem}\label{devissage_crystal_side}
		Let $\cM^{\bullet} \in D^b(\Crys_{W_nX}^{C^r}) \cap D^{\geq 0}$. Then there exists an exact triangle \[ \cN^{\bullet} \to \cM^{\bullet} \to \cP^{\bullet} \xto{+1} \] such that $\cN^{\bullet} \in D^b(\Crys_{X}^{C^r}) \cap D^{\geq 0}$ and $\cP^{\bullet} \in D^b(\Crys_{W_{n - 1}X}^{C^r}) \cap D^{\geq 0}$.
	\end{lem}
	\begin{proof}
		By exactness of $\Coh_{W_nX}^{C^r} \to \Crys_{W_nX}^{C^r}$, it is enough to show the result for $\Coh_{W_nX}^{C^r}$. Fix an complex of injectives $\cI^{\bullet}$ quasi--isomorphic to $\cM^{\bullet}$ in $\IndCoh_{W_nX}^{C^r}$, and consider the short exact sequence of complexes \[ 0 \to i_1^{\flat}\cI^{\bullet} \to \cI^{\bullet} \to \cQ^{\bullet} \to 0 \] (recall that $i_1$ denotes the natural closed immersion $X \inj W_nX$).
		
		We may assume that $\cI^j = 0$ for all $j < 0$, so it is automatic that both $i_1^{\flat}\cI^{\bullet}$ and $\cQ^{\bullet}$ are in $D^{\geq 0}$. Now, let $i \geq 0$ such that $\cH^j(\cI^{\bullet}) = 0$ for all $j \geq i$. Then we obtain that \[ \tau_{\leq {i + 1}} (i_1^{\flat}\cI^{\bullet}) \to \tau_{\leq i + 1}(\cI^{\bullet}) = \tau_{\leq i}(\cI^{\bullet}) \to \tau_{\leq i}(\cQ^{\bullet}) \xto{+1}  \] is an exact triangle.
		
		Note that $i_1^{\flat}(\cI^{\bullet})$ is a complex of objects in $\IndCoh_X^{C^r}$, and $\cQ$ is a complex of objects in $\IndCoh_{W_{n - 1}X}^{C^r}$ (it is annihilated by elements of the form $V^{n - 1}(t)$ with $t \in \cO_X$, since $V^n = 0$).
		
		Let us show that $i_1^{\flat}\cI^{\bullet}$ and $\cQ$ have coherent cohomology sheaves. It is enough to show this result for $i_1^{\flat}\cI^{\bullet}$ by the long exact sequence in cohomology, so our goal is to show that the functor $i^!$ perserves $D^+_{\coh}(\IndCoh_{W_nX}^{C^r})$. Since an injective in $\IndCoh_{W_nX}^{C^r}$ is also injective in $\QCoh_{W_nX}$, the diagram
		 \[ \begin{tikzcd}
			D^+(\IndCoh_{W_nX}^{C^r}) \arrow[d, "i_1^!"] \arrow[rr] &  & D^+(\QCoh_{W_nX}) \arrow[d, "i_1^!"] \\
			D^+(\IndCoh_{X}^{C^r}) \arrow[rr]                       &  & D^+(\QCoh_X)                        
		\end{tikzcd} \] commutes. Hence, the result follows from \stacksproj{0A79}. 
	
		Hence, we have proven that $\tau_{\leq i + 1}(\cI^{\bullet}) \in D^b_{\coh}(\IndCoh_{W_nX}^{C^r})$, so by \cite[Theorem 2.9.1]{Bockle_Pink_Cohomological_Theory_of_crystals_over_function_fields} there exists $\cM^{\bullet} \in D^b(\Coh_{W_nX}^{C^r})$ such that $\cM^{\bullet} \cong \tau_{\leq i + 1}(\cI^{\bullet})$. Similarly, there exists $\cP^{\bullet} \in D^b(\Coh_{W_{n - 1}X})$ quasi--isomorphic to $\cQ^{\bullet}$. Since $\cI^{\bullet} \cong \cM^{\bullet}$, we deduce that the exact triangle \[ \cN^{\bullet} \to \cM^{\bullet} \to \cP^{\bullet} \xto{+1} \] satisfies the required hypotheses.
	\end{proof}

	\begin{lem}\label{devissage_perverse_side}
		Let $\cF^{\bullet} \in D^b_c(X_{\et}, W_n(\bF_q)) \cap {}^pD^{\leq 0}$. Then there exists an exact triangle \[ \cG^{\bullet} \to \cF^{\bullet} \to \cH^{\bullet} \xto{+1} \] with $\cG^{\bullet} \in D^b_c(X_{\et}, W_{n - 1}(\bF_q))$ and $\cH^{\bullet} \in D^b(X_{\et}, \bF_q)$.
	\end{lem}
	\begin{proof}
		The proof will be very similar to that of \autoref{devissage_crystal_side}, replacing the use of $i_1^!$ by that of $- \otimes_{W_n{\bF_q}}^L \bF_q$. By \stacksproj{0960}, we may assume that $\cF^{\bullet}$ is a bounded above complex of flat constructible $W_n(\bF_q)$--sheaves. Consider the exact sequence of complexes \[ 0 \to p\cF^{\bullet} \to \cF^{\bullet} \to \cF^{\bullet} \otimes_{W_n(\bF_q)} \bF_q \to 0. \] By our discussion above, we obtain that $\cF^{\bullet}/p = \cF^{\bullet} \otimes_{W_n(\bF_q)} \bF_q \in D^-_c(X_{\et}, \bF_q)$, and that $p\cF^{\bullet} \in D^-_c(X_{\et}, W_{n - 1}(\bF_q))$.
		
		In order to apply the same truncation argument as in the proof of \autoref{devissage_crystal_side} (with perverse truncations this time) and conclude the proof, it is enough to show that both $\cF^{\bullet}/p$ and  $p\cF^{\bullet}$ live in ${}^pD^{\leq 0}$. Fix $x \in X$ and $j \in \bZ$. Applying $- \otimes^L_{W_n(\bF_q)} \bF_q$ to the triangle \[ \cH^j(\cF^{\bullet})[-j] \to \tau_{\geq j}(\cF^{\bullet}) \to \tau_{\geq j + 1}(\cF^{\bullet}) \xto{+1} \] gives in particular an exact sequence \[ \dots \to \cH^j(\cF^{\bullet})/p \to \cH^j(\cF^{\bullet}/p) \to \cH^{j}(\tau_{\geq j + 1}(\cF^{\bullet}) \otimes^L_{W_n(\bF_q)} \bF_q) \to 0. \]
		
		In particular, if $j > -\dim\overline{\{x\}}$, then by assumption we know that $\tau_{\geq j + 1}(\cF^{\bullet})_{\overline{x}} = 0$ and $\cH^j(\cF^{\bullet})_{\overline{x}} = 0$, so $\cH^j(\cF^{\bullet}/p)_{\overline{x}} = 0$. This shows that $\cF^{\bullet}/p \in {}^pD^{\leq 0}$. Furthermore, the above exact sequence (with $j$ replaced by $j - 1$) shows that the cokernel $\cE$ of $\cH^{j - 1}(\cF^{\bullet}) \to \cH^{j - 1}(\cF^{\bullet}/p)$ is a submodule of $\cH^{j - 1}(\tau_{\geq j}(\cF^{\bullet}) \otimes^L_{W_n(\bF_q)} \bF_q)$, and hence $\cE_{\overline{x}} = 0$ given that $j > -\dim\overline{\{x\}}$. Thanks to the long exact sequence \[ \dots \cH^{j - 1}(\cF^{\bullet}) \to \cH^{j - 1}(\cF^{\bullet}/p) \to \cH^j(p\cF^{\bullet}) \to \cH^j(\cF^{\bullet}) \to \dots, \] we have an exact sequence \[ 0 \to \cE \to \cH^j(p\cF^{\bullet}) \to \cH^j(\cF^{\bullet}) \to \dots. \] Given that $\cE_{\overline{x}} = \cH^j(\cF^{\bullet})_{\overline{x}} = 0$, we conclude that $\cH^j(p\cF^{\bullet})_{\overline{x}} = 0$. We have proven that also $p\cF^{\bullet} \in {}^pD^{\leq 0}$, so the proof is complete.
	\end{proof}

	\begin{thm}\label{main_thm_final}
		Let $X$ be a Noetherian, $F$--finite, semi--separated $\bF_q$--scheme with a $W_n$--unit dualizing complex $W_n\omega_X^{\bullet}$, and associated duality functor $\bD$. Then up to shifting $W_n\omega_X^{\bullet}$ on each connected component of $X$, the equivalence $\Sol \circ \: \bD$ sends the canonical t--structure on $D^b(\Crys_{W_nX}^{C^r})^{op}$ to the perverse t--structure on $D^b_c(X_{\et}, W_n(\bF_q))$, hence inducing an equivalence of categories \[ (\Crys_{W_nX}^{C^r})^{op} \cong \Perv_c(X_{\et}, W_n(\bF_q)). \] Furthermore, the equivalence $\Sol \circ \: \bD$ commutes with derived proper pushforwards, and with the inclusions $D^b(\Crys_{W_mX}^{C^r})^{op} \to D^b(\Crys_{W_nX}^{C^r})^{op}$ and $D^b_c(X_{\et}, W_n(\bF_q))$ for $1 \leq m \leq n$.
	\end{thm}
	\begin{proof}
		Let us shift $W_n\omega_X^{\bullet}$ on each connected component of $X$, so that the dimension function associated to the dualizing complex $i_1^!W_n\omega_X^{\bullet}$ on $X$ is $x \mapsto -\dim\overline{\{x\}}$ (see the statement and proof of \cite[Lemma 5.2.2]{Baudin_Duality_between_perverse_sheaves_and_Cartier_crystals}).
		
		By \autoref{main thm Bhatt Lurie}, \autoref{main thm duality} and \autoref{comp duality and pushforwards}, the functor $\Sol \circ \: \bD$ is indeed an equivalence of categories preserving derived proper pushforwards and the inclusions at the end of the statement. Given that it is an equivalence, we have a priori two t--structures on $D^b_c(X_{\et}, W_n(\bF_q))$: $(\Sol(\bD(D^{\geq 0})), \Sol(\bD(D^{\leq 0})))$ and $({}^pD^{\leq 0}, {}^pD^{\geq 0})$. Since one of the subcategories defining a t--structure completely determines the other one, it is enough to show that $\Sol(\bD(D^{\geq 0})) = {}^pD^{\leq 0}$ to obtain that these two t--structures are equal. We will show the result by induction on $n \geq 1$. This is true for $n = 1$ by \cite[Theorem 5.2.7]{Baudin_Duality_between_perverse_sheaves_and_Cartier_crystals}, so now assume that $n \geq 2$. 
		
		Let $\cM^{\bullet} \in D^{\geq 0}$ in $D^b(\Crys_{W_nX}^{C^r})$. By \autoref{devissage_crystal_side}, there exists an exact triangle \[ \cN^{\bullet} \to \cM^{\bullet} \to \cP^{\bullet} \xto{+1} \] with $\cN^{\bullet} \in D^b(\Crys_X^{C^r})$ and $\cP^{\bullet} \in D^b(\Crys_{W_{n - 1}X}^{C^r})$, with both $\cN^{\bullet}$ and $\cP^{\bullet}$ in $D^{\geq 0}$. Since $\bD$ commutes with $D^b(\Crys_{W_mX}^{C^r}) \to D^b(\Crys_{W_nX}^{C^r})$ for $1 \leq m \leq n$ (see \autoref{main thm duality}) and the analogue statement holds for $\Sol$ (see \autoref{main thm Bhatt Lurie}), we deduce by the induction hypothesis that both $\Sol(\bD(\cN^{\bullet}))$ and $\Sol(\bD(\cP^{\bullet}))$ live in ${}^pD^{\leq 0}$. Since we have an exact triangle \[ \Sol(\bD(\cP^{\bullet})) \to \Sol(\bD(\cM^{\bullet})) \to \Sol(\bD(\cN^{\bullet})) \xto{+1}, \] we conclude that $\Sol(\bD(\cM^{\bullet})) \in {}^pD^{\leq 0}$ by the long exact sequence in perverse cohomology sheaves. 
		
		By the exact same argument (this time using \autoref{devissage_perverse_side}), we see that $(\Sol \circ \bD)^{-1}$ sends ${}^pD^{\leq 0}$ to $D^{\geq 0}$, and hence \[ (\Sol \circ \:\bD)(D^{\geq 0}) = {}^pD^{\leq 0}. \]
	\end{proof}

	\bibliographystyle{alpha}
	\bibliography{bibliography}

\newcommand{\etalchar}[1]{$^{#1}$}
\begin{thebibliography}{KTT{\etalchar{+}}24}

\bibitem[Bau23]{Baudin_Duality_between_perverse_sheaves_and_Cartier_crystals}
J.~Baudin.
\newblock Duality between {C}artier crystals and perverse
  $\mathbb{F}_p$-sheaves, and application to generic vanishing.
\newblock {\em arXiv e-print: arXiv:2306.05378v2}, 2023.
\newblock Available at
  \href{https://arxiv.org/abs/2306.05378}{arXiv:2306.05378}.

\bibitem[Bau25]{Baudin_Witt_GR_vanishing_and_applications}
J.~Baudin.
\newblock A {G}rauert--{R}iemenschneider vanishing theorem for {W}itt canonical
  sheaves.
\newblock {\em To appear}, 2025.

\bibitem[BB11]{Blickle_Bockle_Cartier_modules_finiteness_results}
M.~Blickle and G.~B\"{o}ckle.
\newblock Cartier modules: finiteness results.
\newblock {\em J. Reine Angew. Math.}, 661:85--123, 2011.

\bibitem[BB13]{Blickle_Bockle_Cartier_crystals}
M.~Blickle and G.~B\"{o}ckle.
\newblock Cartier crystals.
\newblock {\em arXiv e-print: arXiv:1309.1035v1}, 2013.
\newblock Available at \href{https://arxiv.org/abs/1309.1035}{arXiv:1309.1035}.

\bibitem[BBK23]{Baudin_Bernasconi_Kawakami_Frobenius_GR_fails}
J.~Baudin, F.~Bernasconi, and T.~Kawakami.
\newblock The {F}robenius--stable version of the {G}rauert--{R}iemenschneider
  vanishing theorem fails.
\newblock {\em arXiv e-print: arXiv:2312.13456v3}, 2023.
\newblock Available at
  \href{https://arxiv.org/abs/2312.13456}{arXiv:2312.13456}.

\bibitem[BBL{\etalchar{+}}23]{Bhatt&Co_Applications_of_perverse_sheaves_in_commutative_algebra}
B.~Bhatt, M.~Blickle, G.~Lyubeznik, A.~K. Singh, and W.~Zhang.
\newblock Applications of perverse sheaves in commutative algebra.
\newblock {\em arXiv e-print, arXiv:2308.03155v1}, 2023.
\newblock Available at
  \href{https://arxiv.org/abs/2308.03155}{arXiv:2308.03155}.

\bibitem[BF22]{Blickle_Fink_Cartier_crystals_have_finite_global_dimension}
M.~Blickle and D.~Fink.
\newblock Cartier crystals have finite global dimension.
\newblock {\em arXiv e-print: arXiv:2211.11466v1}, 2022.
\newblock Available at
  \href{https://arxiv.org/abs/2211.11466}{arXiv:2211.11466}.

\bibitem[BL19]{Bhatt_Lurie_RH_corr_pos_char}
B.~Bhatt and J.~Lurie.
\newblock A {R}iemann-{H}ilbert correspondence in positive characteristic.
\newblock {\em Camb. J. Math.}, 7(1-2):71--217, 2019.

\bibitem[Bli13]{Blickle_Test_ideals_via_algebras_of_pe_linear_maps}
M.~Blickle.
\newblock Test ideals via algebras of {$p^{-e}$}-linear maps.
\newblock {\em J. Algebraic Geom.}, 22(1):49--83, 2013.

\bibitem[BP09]{Bockle_Pink_Cohomological_Theory_of_crystals_over_function_fields}
G.~B\"{o}ckle and R.~Pink.
\newblock {\em Cohomological theory of crystals over function fields}, volume~9
  of {\em EMS Tracts in Mathematics}.
\newblock European Mathematical Society (EMS), Z\"{u}rich, 2009.

\bibitem[CRS23]{Carvajal_Rojas_Stabler_On_the_behavior_of_F_signature_uner_finite_covers}
J.~Carvajal-Rojas and A.~St\"{a}bler.
\newblock On the behavior of {$F$}-signatures, splitting primes, and test
  modules under finite covers.
\newblock {\em J. Pure Appl. Algebra}, 227(1):Paper No. 107165, 38, 2023.

\bibitem[Eji24]{Ejiri_Direct_images_of_pluricanonical_bundles_and_Frobenius_stable_canonical_rings_of_fibers}
S.~Ejiri.
\newblock Direct images of pluricanonical bundles and {F}robenius stable
  canonical rings of fibers.
\newblock {\em Algebr. Geom.}, 11(1):71--110, 2024.

\bibitem[EK04]{Emerton_Kisin_Riemann-Hilbert_correspondence}
M.~Emerton and M.~Kisin.
\newblock The {R}iemann-{H}ilbert correspondence for unit {$F$}-crystals.
\newblock {\em Ast\'{e}risque}, (293):vi+257, 2004.

\bibitem[Gab04]{Gabber_notes_on_some_t_structures}
O.~Gabber.
\newblock Notes on some {$t$}-structures.
\newblock In {\em Geometric aspects of {D}work theory. {V}ol. {I}, {II}}, pages
  711--734. Walter de Gruyter, Berlin, 2004.

\bibitem[GR04]{Gabber_Ramero_Almost_ring_theory}
O.~Gabber and L.~Ramero.
\newblock {\em Foundations for almost ring theory -- {R}elease 7.5}.
\newblock 2004.
\newblock Available at
  \href{https://arxiv.org/abs/math/0409584}{arXiv:0409584}.

\bibitem[HP22]{Hacon_Pat_GV_Geom_Theta_Divs}
C.~D. Hacon and Zs. Patakfalvi.
\newblock Generic vanishing in characteristic {$p>0$} and the geometry of theta
  divisors.
\newblock {\em Boll. Unione Mat. Ital.}, 15(1-2):215--244, 2022.

\bibitem[HS93]{Huneke_Sharp_Bass_numbers_of_local_coh_modules}
C.~L. Huneke and R.~Y. Sharp.
\newblock Bass numbers of local cohomology modules.
\newblock {\em Trans. Amer. Math. Soc.}, 339(2):765--779, 1993.

\bibitem[HX15]{Hacon_Xu_On_the_3dim_MMP_in_pos_char}
C.~D. Hacon and C.~Xu.
\newblock On the three dimensional minimal model program in positive
  characteristic.
\newblock {\em J. Amer. Math. Soc.}, 28(3):711--744, 2015.

\bibitem[KTT{\etalchar{+}}24]{Quasi_F_splittings_III}
T.~Kawakami, T.~Takamatsu, H.~Tanaka, J.~Witaszek, F.~Yobuko, and S.~Yoshikawa.
\newblock Quasi-{F}-splittings in birational geometry {III}.
\newblock {\em arXiv e-print: arXiv:2408.01921v1}, 2024.
\newblock Available at
  \href{https://arxiv.org/abs/2408.01921}{arXiv:2408.01921}.

\bibitem[Kun76]{Kunz_Noetherian_Rings_of_char_p}
E.~Kunz.
\newblock On {N}oetherian rings of characteristic {$p$}.
\newblock {\em Amer. J. Math.}, 98(4):999--1013, 1976.

\bibitem[Lyu97]{Lyubeznik_F_modules_applications_to_local_coh_and_D_mods_in_char_p}
G.~Lyubeznik.
\newblock {$F$}-modules: applications to local cohomology and {$D$}-modules in
  characteristic {$p>0$}.
\newblock {\em J. Reine Angew. Math.}, 491:65--130, 1997.

\bibitem[LZ04]{Langer_Zink_De_Rham_Witt_cohomology_for_a_proper_and_smooth_morphism}
A.~Langer and T.~Zink.
\newblock De {R}ham-{W}itt cohomology for a proper and smooth morphism.
\newblock {\em J. Inst. Math. Jussieu}, 3(2):231--314, 2004.

\bibitem[Mil80]{Milne_Etale_Cohomology}
J.~S. Milne.
\newblock {\em \'{E}tale cohomology}.
\newblock Princeton Mathematical Series, No. 33. Princeton University Press,
  Princeton, N.J., 1980.

\bibitem[Mum08]{Mumford_Abelian_Varieties}
D.~Mumford.
\newblock {\em Abelian varieties}, volume~5 of {\em Tata Institute of
  Fundamental Research Studies in Mathematics}.
\newblock Published for the Tata Institute of Fundamental Research, Bombay; by
  Hindustan Book Agency, New Delhi, 2008.
\newblock With appendices by C. P. Ramanujam and Yuri Manin, Corrected reprint
  of the second (1974) edition.

\bibitem[MW24]{Mattis_Weiss_The_derived_infty_category_of_Cartier_modules}
K.~Mattis and T.~Wei{$\beta$}.
\newblock The derived {$\infty$}-category of {C}artier {M}odules.
\newblock {\em arXiv e-print: arXiv:2410.17102v1}, 2024.
\newblock Available at
  \href{https://arxiv.org/abs/2410.17102}{arXiv:2410.17102}.

\bibitem[Pat18]{Patakfalvi_On_Subadditivity_of_Kodaira_dimension_in_positive_characteristic}
Zs. Patakfalvi.
\newblock On subadditivity of {K}odaira dimension in positive characteristic
  over a general type base.
\newblock {\em J. Algebraic Geom.}, 27(1):21--53, 2018.

\bibitem[Sch11]{Schwede_Test_ideals_in_non_Q_Gor_rings}
K.~Schwede.
\newblock Test ideals in non-{$\Bbb{Q}$}-{G}orenstein rings.
\newblock {\em Trans. Amer. Math. Soc.}, 363(11):5925--5941, 2011.

\bibitem[Sch18]{Schedlmeier_Cartier_crystals_and_perverse_sheaves}
T.~Schedlmeier.
\newblock Cartier crystals and perverse constructible \'etale $p$-torsion
  sheaves.
\newblock {\em arXiv e-print: arXiv:1603.07696v2}, 2018.
\newblock Available at
  \href{https://arxiv.org/abs/1603.07696}{arXiv:1603.07696}.

\bibitem[SGA72]{SGA_4_tome_2}
{\em Th\'{e}orie des topos et cohomologie \'{e}tale des sch\'{e}mas. {T}ome 2}.
\newblock Lecture Notes in Mathematics, Vol. 270. Springer-Verlag, Berlin-New
  York, 1972.
\newblock S\'{e}minaire de G\'{e}om\'{e}trie Alg\'{e}brique du Bois-Marie
  1963--1964 (SGA 4), Dirig\'{e} par M. Artin, A. Grothendieck et J. L.
  Verdier. Avec la collaboration de N. Bourbaki, P. Deligne et B. Saint-Donat.

\bibitem[{Sta}25]{Stacks_Project}
The {Stacks project authors}.
\newblock The {S}tacks project.
\newblock \url{https://stacks.math.columbia.edu}, 2025.

\end{thebibliography}
	
	\Addresses
	
\end{document}